\documentclass[11pt]{amsart}

\usepackage[T1]{fontenc}
\usepackage{amsmath,amssymb,mathtools}
\usepackage{xcolor}
\usepackage[margin=1.15in]{geometry}
\usepackage[colorlinks=true,linkcolor=blue,citecolor=blue,urlcolor=blue]{hyperref}
\hypersetup{
  pdftitle={Failure of analyticity-radius growth in energy-canceling fluid models},
  pdfauthor={Ke Chen, Haina Li, Quoc-Hung Nguyen, Ping Zhang}
}

\newcommand{\T}{\mathbb T}
\newcommand{\R}{\mathbb R}
\newcommand{\Z}{\mathbb Z}
\newcommand{\abs}[1]{\left|#1\right|}
\newcommand{\norm}[1]{\left\|#1\right\|}
\newcommand{\rad}{\operatorname{rad}}

\numberwithin{equation}{section}
\newtheorem{theorem}{Theorem}[section]
\newtheorem{proposition}[theorem]{Proposition}
\newtheorem{lemma}[theorem]{Lemma}

\theoremstyle{remark}
\newtheorem{remark}[theorem]{Remark}

\title[Energy cancellation and Fourier coefficient lower bounds]
{Failure of analyticity-radius growth in energy-canceling fluid models}

\author{Ke Chen}
\address{Department of Applied Mathematics, The Hong Kong Polytechnic
University, Kowloon, Hong Kong, PR China}
\email{k1chen@polyu.edu.hk}

\author{Haina Li}
\address{School of Mathematics and Statistics, Beijing Institute of
Technology, Beijing 100081, China}
\email{lihaina@bit.edu.cn}

\author{Quoc-Hung Nguyen}
\address{State Key Laboratory of Mathematical Sciences,
Academy of Mathematics and Systems Science, Chinese Academy of Sciences,
Beijing 100190, China; Institute of Mathematics, Academy of Mathematics
and Systems Science, the Chinese Academy of Sciences, Beijing 100190,
China}
\email{qhnguyen@amss.ac.cn}

\author{Ping Zhang}
\address{State Key Laboratory of Mathematical Sciences, Academy of Mathematics \& Systems Science, Chinese Academy of Sciences, Beijing 100190,
China; and School of Mathematical Sciences, University of Chinese Academy
of Sciences, Beijing 100049, China}
\email{zp@amss.ac.cn}


\begin{document}
\begin{abstract}
We prove that exact quadratic energy cancellation alone does not force
growth of the spatial analyticity radius.  On \(\T^3\), we construct an
explicit symmetric, translation-invariant, first-order bilinear operator
\(Q\) that preserves the divergence-free class and, for every smooth
real-valued divergence-free vector field
\(v\), satisfies
\[
  \int_{\T^3} Q(v,v)\cdot v\,dx=0.
\]
For every prescribed sufficiently small time \(T>0\), the equation
\(\partial_t u-\Delta u=Q(u,u)\) admits a global smooth, real-valued,
mean-zero, divergence-free solution \(u\) such that
\(\rad(u(0))=\rad(u(T))=1\).  The construction reduces the dynamics on
an invariant cyclic-shear class to viscous Burgers and tunes a Cole--Hopf
heat profile so that its nearest complex zero returns to its initial
distance from the real torus.

For every \(1<\alpha<2\) and every prescribed sufficiently small
\(T>0\), we also construct a symmetric sparse frequency set, its associated
Fourier projection, and trigonometric-polynomial initial data for the
projected dissipative surface quasi-geostrophic equation.  The resulting
unique global smooth solution \(\theta\) has infinite analyticity radius
initially but satisfies \(0<\rad(\theta(T))\leq1\).  An additively separated
Fourier cascade yields coefficientwise exponential lower bounds, while
uniform comparison estimates control the feedback interactions.  Thus
entire analyticity need not persist even from trigonometric-polynomial
data.  Together, the two constructions show that an exact energy identity
alone does not determine the frequency geometry governing analyticity-radius
growth.
\end{abstract}
\maketitle
\enlargethispage{4pt}

\section{Introduction}

Spatial analyticity is one expression of parabolic regularization in
viscous fluid equations.  For the incompressible
Navier--Stokes equations, classical Gevrey estimates and quantitative
lower bounds for the analyticity radius were established in
\cite{FoiasTemam1989,GrujicKukavica1998,HerbstSkibsted2011,lemarie2,lemarie1}; more
recent work has refined the small-time behavior in critical and
subcritical Sobolev spaces \cite{LiZhang2026}.  These results use both
the smoothing of the heat semigroup and detailed mapping properties of
the projected transport nonlinearity.

Relevant earlier results include Giga's space--time analyticity theorem
and the abstract nonlinear-evolution framework of Kato and Masuda
\cite{Giga1983,KatoMasuda1986}.  Weighted Fourier and
Gevrey formulations were developed by Biswas and Swanson
\cite{BiswasSwanson2007} and Foias and Temam \cite{FoiasTemam1989}.
More recent quantitative results include
the scale-invariant semilinear parabolic analysis of Chemin, Gallagher,
and the fourth author \cite{CheminGallagherZhang2020}, the estimates for
Navier--Stokes data in \(L^p\) obtained by Hu and the fourth author
\cite{HuZhang2022}, and the fourth author's instantaneous-radius bounds for
\(L^p\) solutions \cite{Zhang2023}.  These works concern the creation,
propagation, or quantitative lower bounds of analyticity radius.  Here
we ask whether energy cancellation alone forces the solution's
analyticity radius to increase strictly from an initially positive value.

One motivation for this question is the cancellation identity obeyed
by the Navier--Stokes nonlinearity:
\[
  \int_{\T^3}
  \mathbb P\operatorname{div}(u\otimes u)\cdot u\,dx=0.
\]
Although this identity controls the total kinetic energy, it contains
no sign information at the level of individual Fourier interactions.
Simplified, dyadic, and averaged fluid models show that preserving the
scaling, differential order, or energy law of an equation need not
preserve all of its dynamical properties; see, for example,
\cite{MontgomerySmith2001,KatzPavlovic2005,Tao2016}.

Tao's averaged
Navier--Stokes construction gives a particularly sharp example
\cite{Tao2016}.  In Tao's construction, the Euler bilinear operator is
averaged using rotations, dilations, and order-zero Fourier
multipliers.  The resulting operator retains the exact energy
cancellation and is no stronger from the viewpoint of the standard
harmonic-analysis estimates, yet the associated viscous equation has
a smooth solution that blows up in finite time.  Thus the energy law
and generic mapping estimates do not by themselves determine the
global dynamics.  They leave out algebraic and geometric information
carried by the actual transport nonlinearity.

We therefore ask the following question, which is weaker than the
finite-time blowup problem:

\begin{quote}
Does a Navier--Stokes-type quadratic energy cancellation, without
additional information about the interaction geometry, force a strict
increase of the spatial analyticity radius?
\end{quote}

Our first result gives a negative answer for an explicit quadratic
model.  Its first-order nonlinearity preserves divergence-free vector
fields, is the divergence of a quadratic stress followed by the Leray
projection, and satisfies the exact Navier--Stokes-type energy
cancellation.  Nevertheless, at a prescribed positive time the equation
admits a global smooth solution with exponential lower bounds for infinitely
many Fourier coefficients.  The model is anisotropic, but it has three nontrivial
components, and the constructed solution depends on all three spatial
variables.  We obtain the solution within an invariant cyclic family of
viscous Burgers shears.  Within this family, a two-frequency Cole--Hopf profile
is tuned so that a nearest complex zero returns to the same distance
from the real torus at the prescribed time.  The return of this zero gives
the exact Fourier lower bound needed for an upper bound on the
analyticity radius.  Theorem~\ref{thm:main} shows that the radius is
exactly one at times \(0\) and \(T\).

Our second result concerns the dissipative surface quasi-geostrophic
(SQG) equation.  The SQG equation was introduced in connection with
frontogenesis in \cite{ConstantinMajdaTabak1994}.  For subcritical
SQG, Dong and Li proved spatial analyticity for initial data in the
scaling-invariant Lebesgue space, together with decay estimates for
space--time derivatives \cite{DongLi2008}.  In the supercritical range
\(0<\alpha_{\mathrm{diss}}<1\), Li later proved optimal local Gevrey
estimates with Fourier weight
\(e^{ct\Lambda^{\alpha_{\mathrm{diss}}}}\) in critical Sobolev and Besov
spaces, matching the exponent of the linear semigroup
\cite{Li2024SQG}.  Since \(\alpha_{\mathrm{diss}}<1\), this Gevrey
estimate is weaker than spatial analyticity and does not give a positive
lower bound for the usual analyticity radius.
Other regularity and Gevrey results for the dissipative equation include
\cite{KiselevNazarovVolberg2007,CaffarelliVasseur2010,BiswasMartinezSilva2015}.

We study the quadratic nonlinearity of an orthogonally projected
dissipative SQG equation.  The projection preserves the scalar
\(L^2\) cancellation and the usual symmetrized SQG multiplier on every
interaction retained by the projection.  For each prescribed time,
Theorem~\ref{thm:SQG-main} constructs a fixed symmetric sparse frequency set on
which an infinite cascade occurs.  Its integer modes are generated
by a three-state recurrence whose dominant eigenvalue is \(\lambda\), while
the modulus \(\mu\) of the subleading eigenvalue satisfies
\(\mu<\sqrt\lambda\).  Consequently, the input modes in each generating
relation are nearly parallel but have unequal leading lengths.  The
SQG multiplier has size \(\mu^n\approx|k|^\gamma\), where
\(\gamma=\log\mu/\log\lambda<1/2\), while the triangle-inequality deficit decays like
\((\mu^2/\lambda)^n\).  Weighting the cascade by
\(|k|^{\alpha-\gamma}e^{-\rho|k|}\) exactly restores the dissipative
scale \(|k|^\alpha\).

Additive separation eliminates all unintended triads.  The remaining
feedback interactions at the input modes are exponentially smaller after the
weight is introduced.  Uniform upper and lower comparison bounds then survive finite
Galerkin truncation and pass to the limit, yielding a unique global smooth
solution.  For each prescribed sufficiently small time, the construction
starts from finitely many Fourier modes and produces a solution with finite,
strictly positive analyticity radius at that time.

The paper contains two main results:
\begin{enumerate}
\item We construct a Navier--Stokes-type quadratic model with an exact
\(L^2\) cancellation and use an explicit complex zero that returns to its
initial distance to prove that the analyticity radius need not increase strictly.
\item For the projected dissipative SQG equation, we construct a sparse
Fourier projection for which a global smooth solution evolves from
trigonometric-polynomial data and has finite, positive analyticity radius at
the prescribed time.
\end{enumerate}
Both give rigorous coefficientwise Fourier lower bounds for global smooth solutions of
their respective models.

\subsection*{Notation}

For \(d\in\{1,2,3\}\), let
\(\T^d=(\R/2\pi\Z)^d\).  If \(f\) is a scalar- or vector-valued
function on \(\T^d\), we use the Fourier convention
\[
  f(x)=\sum_{k\in\Z^d}\widehat f(k)e^{ik\cdot x},
  \qquad
  \widehat f(k)=\frac1{(2\pi)^d}
  \int_{\T^d}f(x)e^{-ik\cdot x}\,dx.
\]
Here and below, \(\abs{k}\) denotes the Euclidean norm.  The operator
\(\Lambda=(-\Delta)^{1/2}\) is the Fourier multiplier with symbol
\(\abs{k}\).  We write \(A\lesssim B\) when \(A\leq CB\), and
\(A{\approx} B\) when both \(A\lesssim B\) and \(B\lesssim A\).  Unless
stated otherwise, the implicit constants may depend on fixed equation
parameters and the mode geometry, but not on the large generation index or
the Galerkin truncation.

We use \(\R^+=[0,\infty)\).  For vector-valued functions,
\(\abs{\widehat f(k)}\) denotes the Euclidean norm of the Fourier
coefficient.

On the torus, exponential decay of the Fourier coefficients is the
spectral signature of spatial analyticity.  For a scalar- or vector-valued
function \(u\in L^2(\T^d)\), we quantify the largest admissible decay rate by
\begin{equation}\label{eq:radius}
  \rad(u)
  =
  \sup\left\{
  \sigma\geq0:
  \sum_{k\in\Z^d}
  e^{2\sigma\abs{k}}\abs{\widehat u(k)}^2<\infty
  \right\}.
\end{equation}
Thus \(\rad(u)\) is the supremal exponential Fourier weight allowed in
the \(L^2\)-based analytic norm.  When the underlying torus is clear, we
abbreviate \(\norm{f}_{L^2(\T^d)}\) as \(\norm{f}_2\).  All identities below are first
established for smooth real-valued solutions, for which the indicated
integrations by parts and Fourier rearrangements are justified.

\subsection*{Main results}

We now state the two main results: the vector-model theorem and the
sparse projected-SQG cascade theorem.

To isolate what follows from energy cancellation without the full
Navier--Stokes interaction geometry, we use the following symmetric
bilinear operator on vector fields over \(\T^3\):
\[
  Q(u,v)
  =
  -\frac12\mathbb P
  \sum_{j=1}^3
  e_{j+1}\partial_j(u_{j+1}v_{j+1}),
\]
where \(e_1,e_2,e_3\) are the standard basis vectors of \(\R^3\),
\(\mathbb P\) is the Leray projection, and the indices are taken modulo
\(3\); thus \(e_4=e_1\), \(u_4=u_1\), and \(v_4=v_1\).  This explicit
operator is translation invariant, homogeneous of first order in
frequency, and satisfies
\begin{equation}\label{eq:main-cancellation}
  \int_{\T^3}Q(v,v)\cdot v\,dx=0
\end{equation}
for every real-valued smooth divergence-free vector field \(v\).  This
identity is proved in Proposition~\ref{prop:energy} below.

We consider the corresponding semilinear parabolic system
\[
  \partial_t u-\Delta u=Q(u,u),
  \qquad \operatorname{div}u=0,\quad (t,x)\in \R^+\times\T^3.
\]
For the solution constructed below, spatial analyticity follows directly
from the positive Cole--Hopf heat profile in Section~\ref{sec:profile};
at the two endpoint times, the Fourier formula in Section~\ref{sec:Fourier}
determines the radius exactly.  The following theorem shows that this
radius need not increase strictly.  Its finite value at a positive time
also rules out entire analyticity and ultra-analytic regularity for this
solution.

\begin{theorem}
\label{thm:main}
There exists \(T_0>0\) such that, for every \(0<T<T_0\), the
equation above admits a real-valued, mean-zero, divergence-free global smooth solution
\(u\) satisfying the energy identity
\begin{equation}\label{eq:main-vector-energy}
  \frac12\norm{u(t)}_{L^2}^2
  +
  \int_0^t\norm{\nabla u(s)}_{L^2}^2\,ds
  =
  \frac12\norm{u(0)}_{L^2}^2
  \qquad(t\geq0).
\end{equation}
Moreover, the analyticity radius of \(u\) at time \(T\) is not larger
than its initial value; more precisely,
\begin{equation}\label{eq:same-radius}
  \rad(u(0))=\rad(u(T))=1.
\end{equation}
\end{theorem}

For a symmetric set
\(\Gamma=-\Gamma\subset\Z^2\setminus\{0\}\), let
\(\Pi_\Gamma\) be the orthogonal Fourier projection onto \(\Gamma\), that is,
\(\Pi_\Gamma f(x)=\sum_{k\in\Gamma}\widehat f(k)e^{ik\cdot x}\).
Given \(1<\alpha<2\), we consider the projected SQG equation
\begin{equation}\label{eq:projected-SQG}
  \partial_t\theta+\Lambda^\alpha\theta=\mathcal Q_\Gamma(\theta,\theta),
  \qquad (t,x)\in \R^+\times\T^2.
\end{equation}
Here
\(\nabla^\perp=(-\partial_2,\partial_1)\) and 
the quadratic nonlinearity is defined by
\[
  \mathcal Q_\Gamma(h,h)
  =-
  \Pi_\Gamma
  \left[
  \left(\nabla^\perp\Lambda^{-1}\Pi_\Gamma h\right)
  \cdot\nabla\Pi_\Gamma h
  \right].
\]
The inverse multiplier is well defined here because
\(0\notin\Gamma\).

For every symmetric set \(\Gamma\),
Proposition~\ref{prop:SQG-energy} shows that the projected equation
preserves \(\operatorname{Ran}\Pi_\Gamma\) and satisfies the \(L^2\)
energy identity.

For the sparse projection and initial data constructed below, spatial
analyticity on \([0,T]\) follows from the coefficientwise upper bound in
Lemma~\ref{lem:barrier}, proved directly for the projected equation.
The following theorem shows that even trigonometric-polynomial initial
data can evolve into a solution with finite, positive analyticity radius.
Thus entire analyticity need not be preserved, and ultra-analytic
regularity need not hold at positive times.

\begin{theorem}
\label{thm:SQG-main}
Let \(1<\alpha<2\).  There exists \(T_\alpha>0\) such that, for every
\(0<T<T_\alpha\), one can choose a symmetric set
\(\Gamma=\Gamma_{\alpha,T}\subset\Z^2\setminus\{0\}\) and real-valued, mean-zero
trigonometric-polynomial initial data for which
\eqref{eq:projected-SQG} has a unique global smooth solution satisfying
the exact energy identity
\[
  \frac12\norm{\theta(t)}_{L^2}^2
  +\int_0^t\norm{\Lambda^{\alpha/2}\theta(s)}_{L^2}^2\,ds
  =\frac12\norm{\theta(0)}_{L^2}^2
  \qquad(t\geq0).
\]
Moreover, the analyticity radius of \(\theta\) is infinite initially
but finite and positive at time \(T\); more precisely,
\[
  \rad(\theta(0))=\infty,
  \qquad
  0<\rad(\theta(T))\leq1.
\]
\end{theorem}

The initial datum has finite Fourier support, while at time \(T\) the
solution satisfies a coefficientwise lower bound along an infinite lacunary
sequence.  Thus Theorem~\ref{thm:SQG-main} shows that the projected SQG
evolution itself generates such a bound.

\subsection*{Relation to previous work}

Classical analyticity results for the Navier--Stokes equations prove
finiteness of exponentially weighted norms and hence lower
bounds for the analyticity radius; see
\cite{FoiasTemam1989,GrujicKukavica1998,HerbstSkibsted2011,lemarie2,lemarie1,
BaeBiswasTadmor2012}.  Li and the fourth author \cite{LiZhang2026} recently
obtained refined small-time lower bounds for a class of generalized
three-dimensional Navier--Stokes equations in critical and
subcritical Sobolev spaces.  These results quantify analyticity
created from Sobolev initial data.  They do not assert that the
radius of a solution with analytic initial data increases strictly
between the initial time and an arbitrarily prescribed sufficiently small
positive time.  Theorem~\ref{thm:main}
addresses this different endpoint question by producing explicit
coefficientwise Fourier lower bounds at both \(t=0\) and \(t=T\).
Related upper bounds for the maximal analyticity radius of regular
Navier--Stokes solutions were studied by Biswas and Foias
\cite{BiswasFoias2014}.

The relation between complex singularities and spatial analyticity for
viscous Burgers has been studied in detail; see Senouf
\cite{Senouf1997} and Weideman \cite[Section~3]{Weideman2022}.
For periodic Burgers solutions, the latter describes complex poles that
approach the real axis and subsequently recede, giving a contracting and
then expanding analyticity strip.  In Theorem~\ref{thm:main}, we combine
this classical mechanism with an explicit divergence-free vector model
whose energy cancellation holds for every smooth solenoidal input.
The two-frequency heat profile is tuned to a prescribed sufficiently
small time \(T\), giving the exact equality
\(\rad(u(0))=\rad(u(T))=1\) and explicit coefficientwise Fourier lower
bounds at both endpoints.

The constructions of Montgomery-Smith
\cite{MontgomerySmith2001}, Katz and Pavlovi\'c
\cite{KatzPavlovic2005}, and Tao \cite{Tao2016} demonstrate that
energy cancellation and standard differential-order estimates do not
determine the dynamics of a modified fluid equation.  Our
conclusion is different: the vector model admits a global smooth
solution with the same analyticity radius at two prescribed endpoint
times.

Gevrey regularity estimates for dissipative SQG are known in several
regimes; see
\cite{BaeBiswas2015,BiswasMartinezSilva2015,Li2024SQG}.  These results
control the full solution in an exponentially weighted norm.  Boundary
geometry can change the Fourier decay: for subcritical SQG on a square
with mixed boundary conditions, Chernov and Li obtained estimates that
are exponential in one Fourier index but only algebraic in the other
\cite{ChernovLi2012}.  Theorem~\ref{thm:SQG-main} has a different
purpose: it constructs a projected SQG evolution whose exact Fourier
geometry produces exponential lower bounds for infinitely many coefficients
from finite Fourier support.
The projection is highly sparse and depends on the prescribed time, so
the conclusion is not asserted for the unprojected SQG equation.

\subsection*{Analytic mechanisms behind the constructions}

The coefficientwise Fourier lower bound in the vector model comes from the
complex zeros of the Cole--Hopf heat profile.  For
\(s\in\{0,T\}\), the
factorization
\[
\Phi(s,x)
=
2a_2e^{-4s}
(\cosh1-\cos x)(\cosh\rho_s-\cos x),
\qquad \rho_s>1,
\]
places the nearest zeros at \(x=2\pi m\pm i\).  The Fourier identity
\eqref{eq:Fourier-identity} then gives
\[
\widehat g(s,n)
=
2i\left(e^{-n}+e^{-\rho_s n}\right),
\qquad n\geq1.
\]
The upper bound for the endpoint analyticity radius therefore comes from
an explicit nonvanishing family of Fourier coefficients, rather than
from an abstract analytic-norm estimate.

For the projected SQG cascade, a key geometric quantity is the
triangle-inequality deficit
\[
\delta(p,q)
=
\abs{p}+\abs{q}-\abs{p+q}.
\]
The three-state recurrence produces input modes of size comparable to
\(\lambda^n\), angle
\(O((\mu/\lambda)^n)\), and unequal leading lengths.  Therefore
\[
 \delta(p,q)\lesssim(\mu^2/\lambda)^n,
 \qquad
 \left|(p^\perp\!\cdot q)(|p|^{-1}-|q|^{-1})\right|
{\approx}\mu^n{\approx} |p+q|^\gamma.
\]
With \(\beta=\alpha-\gamma\), the weight
\(w_k=|k|^\beta e^{-\rho|k|}\) makes each normalized generating
coefficient comparable to \(\lambda^{\alpha n}\), matching the output
dissipation.  Additive
separation reduces the full projected equation to the two generating
relations and exponentially suppressed feedback interactions at the input
modes.  The comparison bounds then propagate this balance through all
generations.

Section~\ref{sec:model} develops the vector construction from the definition
of the model through the proof of Theorem~\ref{thm:main}.  It first
establishes the energy law and the invariant Burgers reduction, then selects
a heat profile whose nearest complex zero returns to its initial distance and
converts the endpoint factorization into an exact Fourier lower bound.  The
projected SQG construction is independent: Section~\ref{sec:SQG} builds the
sparse cascade and proves the comparison bounds and global regularity needed
to pass it to the infinite system.  Section~\ref{sec:scope} discusses the
scope of both constructions and the remaining questions.

\section{A Navier--Stokes-type model: energy cancellation and endpoint analyticity}
\label{sec:model}

\subsection{The model and its energy law}

We use the standard basis and take component indices modulo \(3\); thus
\(e_4=e_1\), \(u_4=u_1\), and similarly for other vector fields.
Let \(\mathbb P\) denote the Leray projection onto divergence-free
vector fields.  Define the symmetric bilinear operator
\begin{equation}\label{eq:Q}
  Q(u,v)
  =
  -\frac12\mathbb P
  \sum_{j=1}^3
  e_{j+1}\partial_j(u_{j+1}v_{j+1}).
\end{equation}
We consider the following equations for \((t,x)\in\R^+\times\T^3\),
\begin{equation}\label{eq:model}
  \partial_t u-\Delta u=Q(u,u),
  \qquad
  \operatorname{div}u=0,
  \qquad
  u|_{t=0}=u^{\rm in}.
\end{equation}
The projection in \eqref{eq:Q} guarantees that the right-hand side is
divergence free.  Thus, once the initial datum is divergence free,
equation \eqref{eq:model} evolves entirely within the solenoidal
subspace.  As in the Navier--Stokes equation, the projection may be
viewed as eliminating the pressure from the evolution.

This nonlinearity has the same differential order as the
Navier--Stokes transport term.  To make its divergence structure
explicit, set
\begin{equation}\label{eq:stress}
  \mathcal S(u,v)
  =
  \sum_{j=1}^3
  u_{j+1}v_{j+1}\,e_{j+1}\otimes e_j.
\end{equation}
With the row-wise convention
\((\operatorname{div}A)_\ell=\sum_{j=1}^3\partial_j A_{\ell j}\),
\begin{equation}\label{eq:Q-stress}
  Q(u,v)
  =
  -\frac12\mathbb P\operatorname{div}\mathcal S(u,v).
\end{equation}
Thus \(Q\) is translation invariant, bilinear, and homogeneous of
first order in frequency.  Before constructing a special solution, we
verify that the model has the exact structural property under
investigation.  The cyclic arrangement of the components yields the
following cancellation and hence the standard \(L^2\) energy law.

\begin{proposition}
\label{prop:energy}
For every real-valued smooth divergence-free vector field \(u\) on
\(\T^3\),
\begin{equation}\label{eq:cancellation}
  \int_{\T^3}Q(u,u)\cdot u\,dx=0.
\end{equation}
Consequently, every smooth solution of \eqref{eq:model} satisfies
\begin{equation}\label{eq:energy}
  \frac12\frac{d}{dt}\norm{u(t)}_{L^2}^2
  +\norm{\nabla u(t)}_{L^2}^2=0.
\end{equation}
\end{proposition}

\begin{proof}
Because \(u\) is divergence free, \(\mathbb Pu=u\).  The
self-adjointness of \(\mathbb P\) therefore allows us to remove
\(\mathbb P\) from the pairing.  Using the definition of \(Q\) and the
chain rule, we obtain
\begin{align*}
  \int_{\T^3}Q(u,u)\cdot u\,dx
  &=
  -\frac12
  \sum_{j=1}^3
  \int_{\T^3}
  \partial_j(u_{j+1}^2)u_{j+1}\,dx\\
  &=
  -\frac13
  \sum_{j=1}^3
  \int_{\T^3}
  \partial_j(u_{j+1}^3)\,dx
  =0.
\end{align*}
The second equality is the chain rule, and periodicity makes the last
integral vanish for each \(j\) separately.
This componentwise cancellation is precisely the role of the cyclic
placement in \eqref{eq:Q}; no cross-term between different indices is
needed.
The Laplacian contributes
\(-\norm{\nabla u}_{L^2}^2\).  Taking the \(L^2\) inner product of
\eqref{eq:model} with \(u\) and using \eqref{eq:cancellation} now
gives \eqref{eq:energy}.
\end{proof}

\begin{remark}[Scope of the cancellation identity]
The two ingredients play different roles.  The cancellation in
Proposition~\ref{prop:energy} holds for every smooth divergence-free input,
as it should for an analogue of the Navier--Stokes energy law.  The specific
solution constructed below provides the endpoint coefficientwise Fourier lower bound that
prevents a strict increase of the analyticity radius.
\end{remark}

\subsection{Reduction to cyclic Burgers shears}
\label{sec:shear}

To construct a solution whose analyticity radius can be tracked exactly,
we identify an invariant class on which the vector equation becomes
explicitly solvable.  Within this class, the cyclic structure reduces the
three-dimensional system to three copies of the same one-dimensional
viscous Burgers equation.

\begin{proposition}[Invariant shear class]
\label{prop:shear}
Let \(g=g(t,x)\) be a real \(2\pi\)-periodic function of one
variable and define
\begin{equation}\label{eq:shear}
  u(t,x_1,x_2,x_3)
  =
  \bigl(g(t,x_3),g(t,x_1),g(t,x_2)\bigr).
\end{equation}
Then \(\operatorname{div}u=0\), and \(u\) solves
\eqref{eq:model} if and only if
\begin{equation}\label{eq:Burgers}
  \partial_t g-\partial_x^2 g
  =
  -\frac12\partial_x(g^2).
\end{equation}
\end{proposition}

\begin{proof}
Each component in \eqref{eq:shear} is independent of its own
coordinate, so the three terms in its divergence vanish separately.
On the other hand,
\[
  \sum_{j=1}^3
  e_{j+1}\partial_j(u_{j+1}^2)
  =
  \bigl(
  \partial_3g^2(t,x_3),
  \partial_1g^2(t,x_1),
  \partial_2g^2(t,x_2)
  \bigr).
\]
The vector field on the right is also divergence free, so the Leray
projection acts as the identity.  Comparing the three
components of \eqref{eq:model} gives precisely \eqref{eq:Burgers}.
The converse follows from the same computation.
\end{proof}

We next linearize the scalar equation by the classical Cole--Hopf transform
\cite{Hopf1950,Cole1951}.  The following lemma records the normalization
that matches the coefficient in \eqref{eq:Burgers}; its positivity
assumption will later ensure that the logarithm defines a global smooth
profile.

\begin{lemma}[Cole--Hopf transform]
\label{lem:ColeHopf}
Suppose that \(\Phi\in C^\infty([0,\infty)\times\T)\) is positive and
\[
  \partial_t\Phi=\partial_x^2\Phi.
\]
Then
\begin{equation}\label{eq:ColeHopf}
  g=-2\partial_x\log\Phi
\end{equation}
is a smooth real-valued solution of \eqref{eq:Burgers}.
\end{lemma}

\begin{proof}
Let \(h=\log\Phi\).  Since \(\Phi\) is positive, this is a smooth
real-valued function, and the heat equation gives
\[
  h_t=\frac{\Phi_t}{\Phi}
  =\frac{\Phi_{xx}}{\Phi}
  =h_{xx}+h_x^2.
\]
Differentiation in \(x\) yields
\(h_{xt}=h_{xxx}+2h_x h_{xx}\).  Substituting
\(g=-2h_x\), \(g_x=-2h_{xx}\), and
\(g_{xx}=-2h_{xxx}\), we find
\[
  g_t-g_{xx}+gg_x=0,
\]
which is \eqref{eq:Burgers}.  This completes the proof of
Lemma~\ref{lem:ColeHopf}.
\end{proof}

\subsection{A returning complex zero}
\label{sec:profile}

We now specify the heat profile.  Two constraints must be met
simultaneously: positivity on the real torus, which keeps the Cole--Hopf
transform globally smooth, and a complex zero at distance one from the real
axis at both endpoint times.  A two-frequency ansatz meets both constraints.

Set
\begin{equation}\label{eq:D}
  D(T)=e^{-T}+e^{-2T}+e^{-3T},
  \qquad
  c=\cosh1,
  \qquad
  c_2=\cosh2.
\end{equation}
The identity \(c_2=2c^2-1\), together with \(c>1\), implies
\[
  D_*:=\frac{1+2c^2}{c_2}<3=D(0).
\]
Since \(D(T)\) is strictly decreasing from \(3\) to \(0\), there is
a unique \(T_0>0\) such that
\begin{equation}\label{eq:T0}
  D(T_0)=D_*.
\end{equation}

Fix \(0<T<T_0\), set \(r=e^{-T}\) and \(D=D(T)\), and define
\begin{equation}\label{eq:a1a2}
  a_2=\frac1{c_2D},
  \qquad
  a_1=\frac{1+D}{cD}.
\end{equation}
Consider
\begin{equation}\label{eq:Phi}
  \Phi(t,x)
  =
  1-a_1e^{-t}\cos x+a_2e^{-4t}\cos(2x).
\end{equation}
By construction, \(\Phi\) satisfies \(\Phi_t=\Phi_{xx}\).  The coefficients in
\eqref{eq:a1a2} are chosen so that the same prescribed root
appears at both \(t=0\) and \(t=T\).  The following factorization makes this
choice explicit and identifies the nearest complex zeros that will
determine the analyticity radius at endpoint times.

\begin{lemma}
\label{lem:factorization}
For each \(s\in\{0,T\}\), there exists a number
\(\rho_s>1\) such that
\begin{equation}\label{eq:factorization}
  \Phi(s,x)
  =
  2a_2e^{-4s}
  (\cosh1-\cos x)(\cosh\rho_s-\cos x).
\end{equation}
In particular, \(\Phi(s,x)>0\) for every
\((s,x)\in\{0,T\}\times\mathbb R\).
\end{lemma}

\begin{proof}
Set \(y=\cos x\).  Since \(\cos(2x)=2y^2-1\), the expression in
\eqref{eq:Phi} becomes the quadratic polynomial
\[
  P_s(y)
  =2\alpha_s y^2-a_1e^{-s}y+(1-\alpha_s),
  \qquad
  \alpha_s=a_2e^{-4s}.
\]
We first check that \(y=c\) is a root at both endpoint times.
At \(s=0\), \eqref{eq:a1a2} gives
\[
  1-a_1c+a_2c_2
  =
  1-\frac{1+D}{D}+\frac1D
  =0.
\]
At \(s=T\), using
\[
  r(1+D)=r+r^2+r^3+r^4=D+r^4,
\]
we obtain
\[
  1-a_1rc+a_2r^4c_2
  =
  1-\frac{r(1+D)}D+\frac{r^4}D
  =0.
\]

Since the leading and constant coefficients of \(P_s\) are
\(2\alpha_s\) and \(1-\alpha_s\), respectively, the product of its
two roots equals \((1-\alpha_s)/(2\alpha_s)\).  Since one root is \(c\),
the second root is
\[
  C_s
  =
  \frac{1-\alpha_s}{2\alpha_s c}.
\]
The restriction \(T<T_0\), or equivalently \(D>D_*\), gives
\[
  a_2
  =
  \frac1{c_2D}
  <
  \frac1{1+2c^2},
\]
and hence
\[
  C_0=\frac{1-a_2}{2a_2c}>c.
\]
The function \((1-\alpha)/(2\alpha c)\) decreases with \(\alpha\),
and \(\alpha_T<a_2\); hence \(C_T>C_0>c\).  We may therefore define
\(\rho_s=\operatorname{arccosh}C_s\), for which \(\rho_s>1\).
Matching the leading coefficient of the quadratic yields
\[
  P_s(y)=2\alpha_s(c-y)(C_s-y),
\]
which is \eqref{eq:factorization} after substituting
\(y=\cos x\), \(c=\cosh1\), and \(C_s=\cosh\rho_s\).  Since
\(\cosh1>1\) and \(\rho_s>1\), both \(\cosh1\) and
\(\cosh\rho_s\) are strictly larger than \(\cos x\) for every real
\(x\).  Both factors in \eqref{eq:factorization} are therefore positive at
the two endpoint times.
\end{proof}

The factorization also explains the role of the complex zero.  The
zeros of \(\cosh1-\cos z\) are
\(z=2\pi m\pm i\), \(m\in\Z\), whereas the zeros contributed by the
second factor lie at imaginary distance \(\rho_s>1\).  Thus the first
factor supplies the nearest complex zeros at both endpoint times.

The factorization at endpoint times does not by itself guarantee positivity at
intermediate times.  The next proposition supplies this missing global
property, ensuring that the Cole--Hopf logarithm remains well defined for
the entire evolution.

\begin{proposition}
\label{prop:positive-Phi}
The function \(\Phi\) in \eqref{eq:Phi} is strictly positive for
every \(t\geq0\) and \(x\in\mathbb R\).  Consequently,
\[
  g(t,x)=-2\partial_x\log\Phi(t,x)
\]
is a global smooth real-valued solution of
\eqref{eq:Burgers}.
\end{proposition}

\begin{proof}
At time zero, Lemma~\ref{lem:factorization} gives a strictly positive
function on the compact real torus.  The maximum principle for the
heat equation gives
\[
  \min_{x\in\T}\Phi(t,x)
  \geq
  \min_{x\in\T}\Phi(0,x)>0,
  \qquad t\geq0.
\]
Thus strict positivity is preserved for every \(t>0\).
The Cole--Hopf transform is consequently well defined for all time,
and Lemma~\ref{lem:ColeHopf} gives the stated Burgers solution.
\end{proof}

\subsection{Fourier coefficients and the analyticity radius}
\label{sec:Fourier}

At either endpoint, the factorization in \eqref{eq:factorization} separates
two exponential decay scales, \(e^{-n}\) and \(e^{-\rho_s n}\).  The
following identity converts those scales into exact Fourier coefficients;
the first will determine the analyticity radius:
\begin{equation}\label{eq:Fourier-identity}
  -2\partial_x\log(\cosh\rho-\cos x)
  =
  -4\sum_{n\geq1}e^{-\rho n}\sin(nx),
  \qquad \rho>0.
\end{equation}
To see this, put \(q=e^{-\rho}\).  Then
\[
  \cosh\rho-\cos x
  =
  \frac{(1-qe^{ix})(1-qe^{-ix})}{2q},
\]
and \(0<q<1\).  The absolutely convergent expansions
\[
  \log(1-qe^{\pm ix})
  =-
  \sum_{n\geq1}\frac{q^n e^{\pm inx}}n
\]
may therefore be differentiated term by term.  Adding the two
derivatives and multiplying by \(-2\) gives
\(-4\sum_{n\geq1}q^n\sin(nx)\), which is
\eqref{eq:Fourier-identity}.

\begin{proof}[Proof of Theorem~\ref{thm:main}]
Lemma~\ref{lem:ColeHopf} and Proposition~\ref{prop:shear} provide two
successive lifts: the Cole--Hopf transform
takes the positive heat profile to a Burgers solution, and the invariant
cyclic shear takes that scalar solution to the vector model.  Specifically,
take \(\Phi\) from \eqref{eq:Phi}, define \(g\) by \eqref{eq:ColeHopf}, and
form \(u\) by \eqref{eq:shear}.  By Proposition~\ref{prop:positive-Phi},
\(g\) is a global smooth Burgers solution, and
Proposition~\ref{prop:shear} then implies that \(u\) is a global smooth
divergence-free solution of
\eqref{eq:model}.  Since \(g\) is the derivative of a periodic
function, both \(g\) and \(u\) have mean zero.
Integrating \eqref{eq:energy} gives
\eqref{eq:main-vector-energy}.

For \(s\in\{0,T\}\), Lemma~\ref{lem:factorization} and
\eqref{eq:Fourier-identity} give
\begin{equation}\label{eq:g-series}
  g(s,x)
  =
  -4\sum_{n\geq1}
  \left(e^{-n}+e^{-\rho_s n}\right)\sin(nx).
\end{equation}
With the convention
\[
  g(x)=\sum_{n\in\Z}\widehat g(n)e^{inx},
\]
we therefore have
\begin{equation}\label{eq:g-coeff}
  \widehat g(s,n)
  =
  2i\left(e^{-n}+e^{-\rho_s n}\right),
  \qquad n\geq1.
\end{equation}
Under the cyclic identification in \eqref{eq:shear}, the coefficient
\(\widehat g(s,n)\) lies in the \(e_{j+1}\) component of
\(\widehat u(s,ne_j)\).  For each
\(s\in\{0,T\}\), the construction therefore gives some
\(\rho_s>1\) such that
\begin{equation}\label{eq:exact-tail}
  -i e_{j+1}\cdot\widehat u(s,ne_j)
  =
  2\left(e^{-n}+e^{-\rho_s n}\right)
  \geq2e^{-n},
  \qquad n\geq1,\quad j=1,2,3.
\end{equation}

It remains to identify the radius.  All nonzero Fourier modes of the
cyclic shear lie on the three coordinate axes.  Because \(\rho_s>1\),
formula \eqref{eq:g-coeff} gives, for \(0\leq\sigma<1\),
\[
  \sum_{n\geq1}
  e^{2\sigma n}\abs{\widehat g(s,n)}^2
  \leq
  C\sum_{n\geq1}e^{-2(1-\sigma)n}
  <\infty.
\]
The negative frequencies obey the same estimate because \(g\) is
real-valued.  Hence the full weighted square sum for \(u\) is finite
for every \(\sigma<1\).  At \(\sigma=1\), however,
\eqref{eq:exact-tail} implies that the weighted coefficients along
each positive coordinate ray do not even tend to zero.  The
definition \eqref{eq:radius} therefore gives
\eqref{eq:same-radius}.
\end{proof}

\section{A projected SQG model: energy cancellation and a sparse Fourier cascade}
\label{sec:SQG}

\subsection{The projected equation and its energy law}

Fix \(1<\alpha<2\).
The SQG construction uses a different mechanism: a sparse projection that
supports an infinite Fourier cascade.  It is logically independent of the
construction in Section~\ref{sec:model}.  We work
with \eqref{eq:projected-SQG}; for mean-zero solutions, taking
\(\Gamma=\Z^2\setminus\{0\}\) recovers the usual dissipative SQG equation,
whereas a sparse set \(\Gamma\) isolates the cascade geometry developed below.

Before constructing the cascade, we verify that the projection retains the
two structural features needed later: invariance of the selected Fourier
range and exact scalar energy cancellation.
\begin{proposition}
\label{prop:SQG-energy}
Suppose that
\(\theta(0)=\Pi_\Gamma\theta(0)\).  Then every real-valued smooth
solution of
\eqref{eq:projected-SQG}, throughout its interval of existence, satisfies
\[
  \theta(t)=\Pi_\Gamma\theta(t)
\]
and
\begin{equation}\label{eq:SQG-energy}
  \frac12\frac{d}{dt}\norm{\theta(t)}_{L^2(\T^2)}^2
  +
  \norm{\Lambda^{\alpha/2}\theta(t)}_{L^2(\T^2)}^2
  =0.
\end{equation}
\end{proposition}

\begin{proof}
The multipliers \(\Lambda^\alpha\) and \(\Pi_\Gamma\) commute, while
the nonlinear term in \eqref{eq:projected-SQG} lies in the range of
\(\Pi_\Gamma\).  After applying \(\mathrm{Id}-\Pi_\Gamma\), the complementary
Fourier modes therefore satisfy the homogeneous fractional heat
equation with zero initial data.  This proves the range invariance.

Set
\[
  v_\theta=\nabla^\perp\Lambda^{-1}\theta.
\]
Then \(\operatorname{div}v_\theta=0\).  Since
\(\theta=\Pi_\Gamma\theta\) and \(\Pi_\Gamma\) is self-adjoint,
\begin{align*}
  \int_{\T^2}
  \Pi_\Gamma(v_\theta\cdot\nabla\theta)\theta\,dx
  &=
  \int_{\T^2}
  (v_\theta\cdot\nabla\theta)\theta\,dx\\
  &=
  \frac12
  \int_{\T^2}
  v_\theta\cdot\nabla(\theta^2)\,dx
  =0.
\end{align*}
The last equality follows from periodicity and
\(\operatorname{div}v_\theta=0\).
The dissipative term contributes
\(\norm{\Lambda^{\alpha/2}\theta}_{L^2}^2\).  Taking the scalar
\(L^2\) inner product of the equation with \(\theta\) proves
\eqref{eq:SQG-energy}.
\end{proof}

For
\[
  Q_{\rm SQG}(\theta,\theta)
  =
  -
  \left(\nabla^\perp\Lambda^{-1}\theta\right)
  \cdot\nabla\theta,
\]
the projected nonlinearity in \eqref{eq:projected-SQG} can be written as
\begin{equation}\label{eq:QGamma-QSQG}
  \mathcal Q_\Gamma(h,h)
  =
  \Pi_\Gamma
  Q_{\rm SQG}(\Pi_\Gamma h,\Pi_\Gamma h).
\end{equation}
In particular, if \(h=\Pi_\Gamma h\), then
\[
  \mathcal Q_\Gamma(h,h)
  =
  \Pi_\Gamma Q_{\rm SQG}(h,h).
\]
Thus the projection restricts the input and output modes but leaves
the SQG multiplier unchanged on interactions retained by the projection.

For an output \(k\in\Gamma\), the ordered contribution of a pair
\((p,q)\), \(p+q=k\), is
\[
  \frac{p^\perp\cdot q}{\abs{p}}
  \widehat f(p)\widehat f(q).
\]
The same unordered pair also occurs in the reverse order, with
coefficient \((q^\perp\cdot p)/\abs{q}\).  Averaging the two
representations gives the symmetrized Fourier formula
\begin{equation}\label{eq:SQG-symbol}
  \widehat {\mathcal Q_\Gamma(f,f)}(k)
  =
  \frac12
  \sum_{\substack{p+q=k\\p,q\in\Gamma}}
  (p^\perp\cdot q)
  \left(
  \frac1{\abs{p}}-\frac1{\abs{q}}
  \right)
  \widehat f(p)\widehat f(q),
  \qquad k\in\Gamma.
\end{equation}
The Fourier coefficient is zero for \(k\notin\Gamma\).
Here we used \(q^\perp\cdot p=-p^\perp\cdot q\).  The symmetrized
multiplier is therefore positive whenever
\[
\abs{p}<\abs{q},
\qquad
p^\perp\cdot q>0,
\]
since both factors in
\[
(p^\perp\cdot q)
\left(
\frac1{\abs{p}}-\frac1{\abs{q}}
\right)
\]
are then positive.

\subsection{The Fourier-mode recurrence and its spectrum}

The two families of integer modes are initialized by
\begin{equation}\label{eq:UV-initial-data}
 V_1=(1,0),\qquad U_2=(1,1),\qquad V_2=(0,2).
\end{equation}
For \(n\geq2\), they are defined recursively by
\begin{align}
 U_{n+1}&=U_n+V_n,\label{eq:U-relation}\\
 V_{n+1}&=U_n+V_{n-1}.\label{eq:V-relation}
\end{align}
Every mode lies in \(\Z^2\), and
\eqref{eq:U-relation}--\eqref{eq:V-relation} are the two generating Fourier
relations for the cascade.

For the spectral analysis, encode the recurrences by the state matrix
\begin{equation}\label{eq:mode-state}
 \mathbf K_n=
 \begin{pmatrix}V_{n-1}\\ U_n\\ V_n\end{pmatrix},
 \qquad
 A=
 \begin{pmatrix}
 0&0&1\\
 0&1&1\\
 1&1&0
 \end{pmatrix}.
\end{equation}
Here \(\mathbf K_n\) is a three-component column whose entries lie in
\(\Z^2\).  The recurrences are equivalent to
\begin{equation}\label{eq:mode-state-recurrence}
 \mathbf K_{n+1}=A\mathbf K_n.
\end{equation}
With \(z_0=(1,1,0)^{\mathsf T}\), the initial state in
\eqref{eq:UV-initial-data} satisfies
\begin{equation}\label{eq:mode-state-lift}
 \mathbf K_2=(z_0,Az_0),
 \qquad
 \mathbf K_n=(A^{n-2}z_0,A^{n-1}z_0).
\end{equation}
Here we identify \(\mathbf K_n\) with the \(3\times2\) matrix whose rows
are three vector entries.  Thus the two columns in
\eqref{eq:mode-state-lift} are their first and second spatial coordinates.

The characteristic polynomial of \(A\) is
\begin{equation}\label{eq:P}
 P(z)=z^3-z^2-2z+1.
\end{equation}
It is irreducible over \(\mathbb Q\) by the rational-root test.  The sign
computations
\[
 \begin{gathered}
 P(-5/4)=-1/64<0<P(-1),\qquad P(0)>0>P(1),\\
 P(9/5)=-1/125<0<P(11/6)=29/216
 \end{gathered}
\]
locate its three roots, which we label so that
\begin{equation}\label{eq:root-intervals}
 -\frac54<\nu<-1,
 \qquad 0<\omega<1,
 \qquad \frac95<\lambda<\frac{11}{6}.
\end{equation}
Put \(\mu=\abs\nu\).  Since
\(\mu^2<25/16<9/5<\lambda\),
\begin{equation}\label{eq:spectral-gap}
 0<\omega<1<\mu<\sqrt\lambda<\lambda.
\end{equation}
Thus \(\lambda\) is the dominant eigenvalue of the recurrence matrix \(A\).
Put
\begin{equation}\label{eq:gamma}
 \gamma=\frac{\log\mu}{\log\lambda}<\frac12.
\end{equation}
We use the reference scale
\begin{equation}\label{eq:Rn}
 R_n:=\lambda^n.
\end{equation}

Throughout, a triad means a relation \(p+q=k\) among three frequencies
\(p,q,k\in\mathbb Z^2\), with \(k\) distinguished as the output and
\(p,q\) regarded as an unordered pair of inputs.

To turn the recurrence into an isolated cascade, we must both rule out
unintended additive relations and quantify the geometry of the intended
ones.  The next lemma provides these facts simultaneously; its estimates
will determine the size of the SQG multiplier and the exponential weights
used in the comparison argument.

\begin{lemma}\label{lem:geometry}
There is an integer \(N_*\) such that the following holds.  For
\[
 \Gamma_+(N_*)
 =\{U_n,V_n:n\geq N_*\},
\]
all modes are distinct, lie in a fixed acute cone, and satisfy
\[
 \abs{U_n}+\abs{V_n}\approx\lambda^n.
\]
Every relation \(p+q=k\) with \(p,q,k\in\Gamma_+(N_*)\) is one of
\eqref{eq:U-relation} or \eqref{eq:V-relation}, with all three modes in
\(\Gamma_+(N_*)\).  This absence of any other such relation is what we call
additive separation.

For either of these two triads, writing \(p+q=k\) and taking \(n\) to be
the larger input generation,
\begin{align}
 \abs{p}&{\approx}\abs{q}{\approx}\abs{k}{\approx}\lambda^n,
 \label{eq:comparable}\\
 \abs{(p^\perp\cdot q)
 (\abs{p}^{-1}-\abs{q}^{-1})}
 &{\approx}\mu^n{\approx}\abs{k}^\gamma,
 \label{eq:multiplier-size}\\
 0\leq\abs{p}+\abs{q}-\abs{k}
 &\leq C(\mu^2/\lambda)^n.
 \label{eq:defect}
\end{align}
The analogous multiplier coefficients obtained by treating \(p\) or \(q\)
as the output are also comparable to \(\mu^n\).
\end{lemma}

\begin{proof}
\noindent\emph{Step 1: spectral form.}
Recall that \(\nu\), \(\omega\), and \(\lambda\) denote, respectively,
the negative root, the smaller positive root, and the larger positive root
of \(P\), and hence are the three eigenvalues of \(A\).
Their locations are given by \eqref{eq:root-intervals}; numerically,
\[
 \nu\approx-1.24698,\qquad
 \omega\approx0.44504,\qquad
 \lambda\approx1.80194.
\]
Since \(\det[z_0,Az_0,A^2z_0]=-1\), the vector \(z_0\) is cyclic.  As
\(P\) is separable, its projection onto each eigenspace is nonzero.  For a
root \(\zeta\) of \(P\), use the right eigenvector
\[
 r(\zeta)=\left(\frac{\zeta-1}{\zeta},1,\zeta-1\right).
\]
With \(u_\zeta=(1,\zeta)\) and
\[
 \varphi_K(\zeta)=
 \begin{cases}
 1,&K=U,\\
 \zeta-1,&K=V.
 \end{cases}
\]
the eigenspace decomposition of \(z_0\) and
\eqref{eq:mode-state-lift} give, for \(K=U,V\),
\begin{equation}\label{eq:mode-spectral-expansion}
 K_n=c_\lambda\varphi_K(\lambda)\lambda^n u_\lambda
     +c_\nu\varphi_K(\nu)\nu^n u_\nu
     +O(\omega^n).
\end{equation}
Here \(O(\omega^n)\) denotes a vector whose Euclidean
norm is bounded by \(C\omega^n\), where \(C\) is independent of \(n\).
Moreover, all
\(c_\zeta\ne0\), and fixed powers of \(\zeta\) have been absorbed into
these constants.
Because \(u_\lambda=(1,\lambda)\) lies in the positive quadrant and
\(\lambda>\mu>\omega\), the leading \(\lambda^n\) term implies, after increasing
\(N_*\), that all modes lie in a fixed acute cone and
\[
 \abs{U_n}+\abs{V_n}{\approx}\lambda^n.
\]

\smallskip
\noindent\emph{Step 2: additive separation.}
Set
\[
 \kappa_U=1,
 \qquad \kappa_V=\lambda-1=:a_*.
\]
The rational bounds in \eqref{eq:root-intervals} give
\begin{equation}\label{eq:generation-reduction-inequalities}
 (\lambda-1)\lambda^2>2,
 \quad \lambda^{-1}<\lambda-1,
 \quad (\lambda-1)\lambda>1,
 \quad \lambda^{-3}<2-\lambda,
 \quad \lambda^{-2}<\lambda(\lambda-1)-1.
\end{equation}
All but the fourth follow directly from \(\lambda>9/5\).  The fourth
follows because \(x^3(2-x)\) decreases on \([9/5,11/6]\) and has value
\(1331/1296>1\) at \(x=11/6\).

Write \(K_n^U=U_n\) and \(K_n^V=V_n\).  Suppose that
\(K_i^a+K_j^b=K_m^c\), where
\(a,b,c\in\{U,V\}\).  Thus \(a\) and \(b\) record the
families of the two input modes, while \(c\) records the family of the
output mode.  Assuming \(i\geq j\), put \(d=i-j\), the generation lag
between the inputs.  The cone and size estimates give
\(\abs{m-i}\leq C_0\) for a fixed constant \(C_0\).
To derive the normalized leading relation explicitly, write
\eqref{eq:mode-spectral-expansion} in the form
\[
 K_h^e=c_\lambda\kappa_e\lambda^h u_\lambda+E_h^e,
 \qquad \abs{E_h^e}\leq C\mu^h,
 \qquad e\in\{U,V\}.
\]
Here the \(\nu^h\) and \(\omega^h\) terms are included in the remainder,
since \(\abs{\nu}=\mu>\omega\).  Substitution into the exact relation
\(K_i^a+K_j^b=K_m^c\) gives
\[
 c_\lambda
 \bigl(\kappa_a\lambda^i+\kappa_b\lambda^j-\kappa_c\lambda^m\bigr)
 u_\lambda=E_m^c-E_i^a-E_j^b.
\]
Take \(L_\lambda(x,y)=x\), so that \(L_\lambda(u_\lambda)=1\).
Applying this linear functional yields
\[
 c_\lambda
 \bigl(\kappa_a\lambda^i+\kappa_b\lambda^j-\kappa_c\lambda^m\bigr)
 =O(\mu^i+\mu^j+\mu^m)=O(\mu^i).
\]
Indeed, \(j\leq i\), \(\mu>1\), and \(\abs{m-i}\leq C_0\) imply
\[
 \mu^j\leq\mu^i,
 \qquad \mu^m\leq\mu^{C_0}\mu^i.
\]
Dividing by \(c_\lambda\lambda^i\), with \(c_\lambda\ne0\), and using
\(j-i=-d\), we obtain
\begin{equation}\label{eq:normalized-leading-relation}
 \kappa_a+\lambda^{-d}\kappa_b-\lambda^{m-i}\kappa_c
 =O((\mu/\lambda)^i).
\end{equation}
The error is uniform even when \(d\) depends on \(i\).
Increase \(N_*\), if necessary, so
that this error is smaller than half of every strict algebraic gap used
below.  The inequalities in
\eqref{eq:generation-reduction-inequalities} then imply
\[
 m\in\{i,i+1\},
 \qquad
 \begin{cases}
 m=i: & c=U,\ a=V,\ 0\leq d\leq2,\\
 m=i+1: & 0\leq d\leq1.
 \end{cases}
\]
Indeed, \(m\leq i-1\) gives output at most \(\lambda^{-1}<a_*\), while
\(m\geq i+2\) gives output at least \(a_*\lambda^2>2\).  Suppose
\(m=i\).  The case \(a=c\) is impossible already in the exact vector
identity, because cancellation of \(K_i^a=K_m^c\) would give
\(K_j^b=0\).  The ordering
\(\kappa_V=a_*<\kappa_U=1\) then forces \((a,c)=(V,U)\), and
\(1-a_*=2-\lambda>\lambda^{-3}\) excludes \(d\geq3\).  If \(m=i+1\),
the inequality \(1+\lambda^{-2}<\lambda a_*\) excludes \(d\geq2\).

It remains to test finitely many leading \(\lambda^i\) coefficients.  Put
\(\sigma=m-i\in\{0,1\}\), the output-generation offset, and define
\[
 \Psi_{\sigma,d}^{a,b,c}(z)
 =\varphi_a(z)+z^{-d}\varphi_b(z)-z^\sigma\varphi_c(z).
\]
By \eqref{eq:mode-spectral-expansion}, a candidate relation satisfies
\[
 0=c_\lambda\lambda^i
 \Psi_{\sigma,d}^{a,b,c}(\lambda)u_\lambda+O(\mu^i).
\]
Thus \(\Psi_{\sigma,d}^{a,b,c}(\lambda)\) is the normalized coefficient
of the leading \(\lambda^i u_\lambda\) component of the candidate relation.
Since \(P\) is the minimal polynomial of \(\lambda\), this leading
coefficient vanishes precisely when the remainder of
\(z^d\Psi_{\sigma,d}^{a,b,c}(z)\) modulo \(P(z)\) is zero.  The finite
polynomial divisions give the following complete list; the entries in the
last column are ordered as \((a,b,d,c)\).
\begin{center}
\begin{tabular}{lll}
\textit{case} & \textit{lag range} & \textit{zero remainders} \\
same-generation output & \(0\leq d\leq2\) & none \\
next-generation output & \(0\leq d\leq1\)
& \((U,V,0,U),(V,U,0,U),(U,V,1,V)\).
\end{tabular}
\end{center}
The three ordered forms correspond to the two relations
\eqref{eq:U-relation}--\eqref{eq:V-relation}.  Every other candidate has
a nonzero \(\lambda^i u_\lambda\) term dominating its \(O(\mu^i)\)
remainder, hence cannot occur for large \(i\).  Increasing \(N_*\) removes
all exceptional generations.  The same leading-order comparison makes the modes
distinct: \(\kappa_U\ne\kappa_V\) separates one generation, and
\(\lambda a_*>1\) separates adjacent generations.

\smallskip
\noindent\emph{Step 3: generating coefficients and the triangle-inequality deficit.}
For vectors \(p=(p_1,p_2)\) and \(q=(q_1,q_2)\) in \(\mathbb R^2\),
we write \(\det(p,q)\) for the determinant of the matrix with columns
\(p\) and \(q\); explicitly,
\[
 \det(p,q):=\det\begin{pmatrix}p_1&q_1\\p_2&q_2\end{pmatrix}
 =p_1q_2-p_2q_1=p^\perp\cdot q,
 \qquad p^\perp=(-p_2,p_1).
\]
In particular, since \(u_\lambda=(1,\lambda)\) and
\(u_\nu=(1,\nu)\),
\[
 \det(u_\lambda,u_\nu)
 =\det\begin{pmatrix}1&1\\\lambda&\nu\end{pmatrix}
 =\nu-\lambda\ne0.
\]
For the input pairs \((U_n,V_n)\) and \((U_n,V_{n-1})\), respectively,
the leading \(\lambda^n\nu^n\) determinant factors are
\[
 \nu-\lambda,
 \qquad
 \frac{\nu-1}{\nu}-\frac{\lambda-1}{\lambda}
 =-\frac{\lambda-\nu}{\lambda\nu}.
\]
Since these factors and \(\det(u_\lambda,u_\nu)=\nu-\lambda\) are nonzero,
in either case
\[
 \abs{\det(p,q)}{\approx}(\lambda\mu)^n.
\]
The normalized leading constants in the length asymptotics of the two
input pairs are \((1,\lambda-1)\) and
\((1,(\lambda-1)/\lambda)\).  Thus the modes are
comparable and
\[
 \left|\frac1{\abs{p}}-\frac1{\abs{q}}\right|{\approx}\lambda^{-n}.
\]
Together with \(\abs{k}{\approx}\lambda^n\), these estimates prove
\eqref{eq:comparable}--\eqref{eq:multiplier-size}, since
\(\mu^n=(\lambda^n)^\gamma\).

Let \(\vartheta_n\) be the angle between the two input vectors \(p\) and
\(q\).  Its cosine is defined by the Euclidean inner product:
\[
 \cos\vartheta_n=\frac{p\cdot q}{\abs{p}\abs{q}}
 =\frac{p_1q_1+p_2q_2}
 {\sqrt{p_1^2+p_2^2}\sqrt{q_1^2+q_2^2}}.
\]
Since the inputs lie in a common acute cone, \(0\leq\vartheta_n<\pi/2\).
Using \(\sin\vartheta_n=\abs{\det(p,q)}/(\abs{p}\abs{q})\), the
determinant estimate gives
\(\vartheta_n=O((\mu/\lambda)^n)\).  The exact identity
\[
 \abs{p}+\abs{q}-\abs{p+q}
 =\frac{2\abs{p}\abs{q}(1-\cos\vartheta_n)}
 {\abs{p}+\abs{q}+\abs{p+q}}
\]
therefore yields
\[
\abs{p}+\abs{q}-\abs{p+q}
 \lesssim\lambda^n\vartheta_n^2
 \lesssim(\mu^2/\lambda)^n,
\]
which proves \eqref{eq:defect}.

\smallskip
\noindent\emph{Step 4: coefficients at the input modes.}
The normalized leading constants for the three mode lengths in the
same-generation and lagged triads are
\[
 (\lambda-1,1,\lambda),
 \qquad
 ((\lambda-1)/\lambda,1,\lambda(\lambda-1)),
\]
respectively.  Each triple is pairwise distinct, so every reciprocal
difference in a feedback coefficient is \({\approx}\lambda^{-n}\).
Multiplication by the determinant
\({\approx}(\lambda\mu)^n\) proves the last assertion.
\end{proof}

Fix \(n_0-1\geq N_*\), and set
\begin{equation}\label{eq:Gamma}
 \Gamma_+=\{U_n,V_n:n\geq n_0-1\},
 \qquad
 \Gamma=\Gamma_+\cup(-\Gamma_+).
\end{equation}
Since the modes \(U_n,V_n\), \(n\geq n_0-1\), are pairwise distinct, this
allows us to define a generation map
\begin{equation}\label{eq:generation-definition}
 {\mathcal{T}_{\mathrm{gen}}}:\Gamma\to\mathbb N,
 \qquad
{\mathcal{T}_{\mathrm{gen}}}(\pm U_n)={\mathcal{T}_{\mathrm{gen}}}(\pm V_n)=n.
\end{equation}
Every signed additive relation in \(\Gamma\) is either a positive relation
from Lemma~\ref{lem:geometry} or one of its rearranged forms.  Indeed, choose a
linear functional strictly positive on the fixed acute cone containing
\(\Gamma_+\).  It excludes any relation with positive modes on one side and
a negative mode on the other.  An all-negative identity becomes
all-positive after multiplication by \(-1\), while every other mixed-sign
identity can be rearranged as \(p+q=k\) with
\(p,q,k\in\Gamma_+\).

\subsection{Duhamel iteration and comparison bounds}

We now convert the geometric estimates of Lemma~\ref{lem:geometry} into
coefficientwise upper and lower bounds.  Three preparations are needed:
a sign convention that makes every generating interaction positive, an
exponential weight that balances the dissipation, and a degree that records
the power of the base amplitude carried by each mode.

For a triad \(p+q=k\), we indicate both the output and the input
frequencies in each interaction coefficient.  The index before the
semicolon denotes the output, and the two indices after it denote the
unordered pair of inputs.  The relations
\(p+q=k\), \(k+(-q)=p\), and \(k+(-p)=q\) have coefficients
\[
 \begin{aligned}
 \mathbf{c}_{k;p,q}
 &=(p^\perp\cdot q)(\abs{p}^{-1}-\abs{q}^{-1}),\\
 \mathbf{c}_{p;k,-q}
 &=(p^\perp\cdot q)(\abs{q}^{-1}-\abs{k}^{-1}),\\
 \mathbf{c}_{q;k,-p}
 &=(p^\perp\cdot q)(\abs{k}^{-1}-\abs{p}^{-1}).
 \end{aligned}
\]
The last two are the feedback coefficients at \(p\) and \(q\).
Exchanging the two inputs leaves each coefficient unchanged.
These coefficients satisfy
\[
 \mathbf{c}_{k;p,q}+\mathbf{c}_{p;k,-q}+\mathbf{c}_{q;k,-p}=0.
\]

Every positive mode has at most one generating triad.  Choose
\(\mathbf{s}_v\in\{-1,1\}\) arbitrarily on generation \(n_0-1\) and at
\(V_{n_0}\).  Next choose the sign at \(U_{n_0}\) using the exceptional
relation
\(U_{n_0}=U_{n_0-1}+V_{n_0-1}\), and then choose all remaining signs
recursively.  More precisely, for each generating triad \(p+q=k\),
once \(\mathbf{s}_p,\mathbf{s}_q\in\{-1,1\}\) have been chosen, set
\[
 \mathbf{s}_k=\operatorname{sgn}(\mathbf{c}_{k;p,q})
 \mathbf{s}_p\mathbf{s}_q.
\]
This procedure is unambiguous because each positive mode has at most one
generating triad and the input generations precede the output generation.
Extend the signs to negative modes by \(\mathbf{s}_{-v}=\mathbf{s}_v\).
Since Lemma~\ref{lem:geometry} gives \(\mathbf{c}_{k;p,q}\ne0\), this
choice yields
\begin{equation}\label{eq:positive-forward-interaction}
 \mathbf{s}_k\mathbf{c}_{k;p,q}\mathbf{s}_p\mathbf{s}_q
 =\abs{\mathbf{c}_{k;p,q}}(\mathbf{s}_p\mathbf{s}_q)^2
 =\abs{\mathbf{c}_{k;p,q}}>0
\end{equation}
for every generating triad.  The subspace defined by real and even Fourier
coefficients,
\(\widehat\theta(-v)=\widehat\theta(v)\in\mathbb R\), is invariant under
the projected equation.  On this subspace, set
\(b_v=\mathbf{s}_v\widehat\theta(v)\).  In the \(b_v\)-variables, every
generating coefficient is positive.  The two input amplitudes, however,
need not have fixed signs before their lower bounds have been established.
The transformed feedback signs are
\(\operatorname{sgn}(\mathbf{c}_{p;k,-q}\mathbf{c}_{k;p,q})\) and
\(\operatorname{sgn}(\mathbf{c}_{q;k,-p}\mathbf{c}_{k;p,q})\); one is positive and
one is negative.

For \(\rho>0\), put
\begin{equation}\label{eq:weight}
 \beta=\alpha-\gamma,
 \qquad
 w_v=\abs{v}^\beta e^{-\rho\abs{v}},
 \qquad
 x_v=b_v/w_v.
\end{equation}
For a generating triad whose larger input generation is \(n\),
Lemma~\ref{lem:geometry} gives
\begin{equation}\label{eq:normalized-generating}
 c R_n^\alpha
 \leq
 \abs{\mathbf{c}_{k;p,q}}\frac{w_p w_q}{w_k}
 \leq C R_n^\alpha.
\end{equation}
Thus the normalized generating coefficient has the same order as the
output dissipation.  In contrast, the feedback coefficient in the equation
at \(p\) satisfies
\begin{equation}\label{eq:normalized-feedback}
 \abs{\mathbf{c}_{p;k,-q}}\frac{w_k w_q}{w_p}
 \leq C R_n^\alpha e^{-c\rho R_n},
\end{equation}
because, with \(\delta(p,q)=\abs{p}+\abs{q}-\abs{k}\),
\[
 \abs{q}+\abs{k}-\abs{p}
 =2\abs{q}-\delta(p,q)\geq cR_n
\]
by \eqref{eq:comparable} and \eqref{eq:defect}.
The analogous estimate holds at \(q\).

We next record the power of the base amplitude produced by the generating
dynamics.  Assign to each positive mode a homogeneity degree \(\ell_v\).
Set \(\ell_v=1\) on the two base generations \(n_0-1,n_0\), and impose
\begin{equation}\label{eq:cascade-degree}
 \ell_k=\ell_p+\ell_q
\end{equation}
on every subsequent generating triad.  Relation \eqref{eq:cascade-degree} is not
imposed on the single exceptional base triad
\begin{equation}\label{eq:exceptional-base-triads}
 U_{n_0}=U_{n_0-1}+V_{n_0-1}.
\end{equation}
All three modes in this triad have homogeneity degree one, and we set
\(\ell_{-v}=\ell_v\) on negative modes.  For \(n\geq n_0\),
define the degree vector
\[
 \boldsymbol\ell_n
 =(\ell_{V_{n-1}},\ell_{U_n},\ell_{V_n})^{\mathsf T}.
\]
Beyond the two base generations, the degree vector obeys
\(\boldsymbol\ell_{n+1}=A\boldsymbol\ell_n\).  Since \(A^4\) has strictly
positive entries, the dominant \(\lambda\)-eigenmode controls the growth of
\(A^n\boldsymbol\ell_{n_0}\), and hence
\begin{equation}\label{eq:cascade-degree-growth}
 \ell_v{\approx}\lambda^{n-n_0}
 \quad\bigl(v\in\{U_n,V_n\}\bigr),
 \qquad
 \ell_p{\approx}\ell_q
\end{equation}
for the two inputs of every generating triad of generation \(n\), with
uniform constants, including those for the finitely many initial
generations.

To initialize the first finitely many generations, we compare the rescaled
dynamics with the triangular system obtained by deleting the feedback and
exceptional interactions.  We first record an abstract finite-dimensional
stability lemma.  
For any finite index set \(\mathcal J\) and vector
\(Z=(Z_v)_{v\in\mathcal J}\), write
\[
 \norm{Z}_{\ell^\infty(\mathcal J)}
 :=\max_{v\in\mathcal J}\abs{Z_v}.
\]
The stability argument in Lemma~\ref{lem:finite-depth-initialization} only
uses the components of the perturbation indexed by its finite coordinate
set; those components may depend on auxiliary coordinates outside that set.

\begin{lemma}[Uniform finite-depth stability of triangular systems]
\label{lem:finite-depth-initialization}
Fix \(L\in\mathbb N\), \(S>0\), \(K_0>0\), and \(M>0\).  Consider a
family of finite coordinate sets \(\mathcal I=\mathcal I(n_0,N)\), indexed
by pairs \((n_0,N)\), and suppose that each \(\mathcal I\) is partitioned
into pairwise disjoint relative generations
\(\mathcal I_r\), \(r=-1,0,\ldots,L\), with
\[
 \mathcal I=\bigcup_{r=-1}^{L}\mathcal I_r.
\]

For each \(v\in\mathcal I\), let
\[
 0\leq D_v\leq K_0,
 \qquad \abs{\mathcal A_v}\leq K_0,
 \qquad \iota_v\in\{0,1\}.
\]
Whenever \(\iota_v=1\) and \(v\in\mathcal I_r\), assume that the two input
coordinates satisfy
\[
 p(v),q(v)\in\bigcup_{r'<r}\mathcal I_{r'}.
\]
Define the triangular reference vector field
\begin{equation}\label{eq:abstract-triangular-system}
 (F_0(Y))_v=
 \begin{cases}
  -D_v Y_v+\mathcal A_v Y_{p(v)}Y_{q(v)},&\iota_v=1,\\
  -D_v Y_v,&\iota_v=0.
 \end{cases}
\end{equation}
Let \(Y^0\) solve
\[
 \dot Y^0=F_0(Y^0),
 \qquad Y^0(0)=Y^{\mathrm{in}},
 \qquad \norm{Y^{\mathrm{in}}}_{\ell^\infty(\mathcal I)}\leq K_0.
\]
For \(0<\varepsilon\leq1\), let \(Y^\varepsilon\) have the same initial
data and satisfy, on \(0\leq s\leq S\),
\begin{equation}\label{eq:abstract-perturbed-system}
 \begin{aligned}
  \dot Y^\varepsilon&=F_0(Y^\varepsilon)+G^\varepsilon,
  &Y^\varepsilon(0)&=Y^{\mathrm{in}},\\
  \sup_{0\leq s\leq S}
  \norm{Y^\varepsilon(s)}_{\ell^\infty(\mathcal I)}&\leq M,\\
  \sup_{0\leq s\leq S}
  \norm{G^\varepsilon(s)}_{\ell^\infty(\mathcal I)}&\leq K_0\varepsilon.
 \end{aligned}
\end{equation}
Then there are constants
\[
 M_*=M_*(L,S,K_0),
 \qquad C_*=C_*(L,S,K_0,M),
\]
independent of \(n_0\), \(N\), and \(\varepsilon\), such that
\begin{align}
 \sup_{0\leq s\leq S}\norm{Y^0(s)}_{\ell^\infty(\mathcal I)}
 &\leq M_*,\label{eq:triangular-uniform-bound}\\
 \sup_{0\leq s\leq S}
 \norm{Y^\varepsilon(s)-Y^0(s)}_{\ell^\infty(\mathcal I)}
 &\leq C_*\varepsilon.
 \label{eq:abstract-initialization-convergence}
\end{align}
\end{lemma}

\begin{proof}
We first bound the reference system by induction over the relative
generations.  At generation \(-1\) there is no generating term, so set
\(M_{-1}=K_0\).  Suppose all generations through \(r-1\) are bounded by
\(M_{r-1}\).  For \(v\in\mathcal I_r\) with \(\iota_v=1\), variation of
constants in \eqref{eq:abstract-triangular-system} gives
\[
 \abs{Y_v^0(s)}
 \leq K_0+K_0\int_0^s
       \abs{Y_{p(v)}^0(\sigma)Y_{q(v)}^0(\sigma)}\,d\sigma
 \leq K_0+SK_0M_{r-1}^2.
\]
For \(\iota_v=0\), the same bound holds with the integral omitted.
The recursion
\[
 M_r=\max\{M_{r-1},K_0+SK_0M_{r-1}^2\},
 \qquad 0\leq r\leq L,
\]
proves
\eqref{eq:triangular-uniform-bound}; the resulting bound is independent of
the initial generation and the truncation.

Put \(E=Y^\varepsilon-Y^0\).  Subtracting the equations coordinatewise gives
\begin{equation}\label{eq:abstract-initialization-difference}
 \begin{aligned}
 \dot E_v={}&-D_v E_v+\mathcal A_v
 \bigl(E_{p(v)}Y_{q(v)}^\varepsilon
       +Y_{p(v)}^0E_{q(v)}\bigr)+G_v^\varepsilon,
 &&\iota_v=1,\\
 \dot E_v={}&-D_v E_v+G_v^\varepsilon,
 &&\iota_v=0,
 \end{aligned}
\end{equation}
In both cases, \(E_v(0)=0\).
On the product of the balls determined by
\eqref{eq:abstract-perturbed-system} and
\eqref{eq:triangular-uniform-bound}, the triangular vector field has the
uniform Lipschitz bound
\[
 \norm{F_0(Y^\varepsilon)-F_0(Y^0)}_{\ell^\infty(\mathcal I)}
 \leq K_0\bigl(1+M+M_*\bigr)
       \norm{E}_{\ell^\infty(\mathcal I)}.
\]
Integrating \eqref{eq:abstract-initialization-difference}, using the forcing
bound, and applying Gronwall's inequality yields
\[
 \sup_{0\leq s\leq S}\norm{E(s)}_{\ell^\infty(\mathcal I)}
 \leq K_0S\exp\!\left(K_0(1+M+M_*)S\right)\varepsilon.
\]
The resulting constants are independent of \(n_0\), \(N\), and
\(\varepsilon\).
\end{proof}

The finite-depth lemma controls a fixed finite number of generations.  The
next result combines this initialization with a Duhamel induction to obtain
coefficientwise upper and lower bounds for the normalized amplitudes
\(x_v(t)\) at every generation.  The constants are uniform in the Galerkin
truncation level \(N\).

\begin{lemma}[Uniform comparison bounds on the sparse cascade]\label{lem:barrier}
Fix \(C_{\rm time}>0\), \(\tau>0\), and
\(\chi\in(\lambda,\lambda^\alpha)\).  There exist constants
\[
 0<\underline c<1<\overline C,
 \qquad 0<\varepsilon_-<\varepsilon_0<\varepsilon_+<1,
\]
depending only on \(\alpha,C_{\rm time},\tau,\chi\) and the fixed mode
recurrence.
For
every \(n_0\) with \(n_0-1\geq N_*\), every
\(1/2\leq\rho\leq1\), and every \(N>n_0\), define \(\Gamma\) by
\eqref{eq:Gamma}.  Use the signed variables
\(b_v=\mathbf{s}_v\widehat\theta(v)\), the weights from
\eqref{eq:weight}, and the amplitude degrees from
\eqref{eq:cascade-degree}.  Set
\begin{equation}\label{eq:finite-truncation}
 \Gamma^{(N)}
 =\{\pm U_n,\pm V_n:n_0-1\leq n\leq N\}.
\end{equation}
All amplitudes below are indexed by \(v\in\Gamma^{(N)}\).  Since
\(x_{-v}=x_v\), it suffices in the proof to consider
\(v\in\Gamma_+\cap\Gamma^{(N)}\).

Prescribe the initial data by
\[
 x_v(0)=
 \begin{cases}
  \varepsilon_0,
  &{\mathcal{T}_{\mathrm{gen}}}(v)\in\{n_0-1,n_0\},\\
  0,&n_0<{\mathcal{T}_{\mathrm{gen}}}(v)\leq N,
 \end{cases}
 \qquad v\in\Gamma^{(N)}.
\]
Let \(x_v(t)\) denote the normalized amplitudes of the corresponding finite
projected solution, and define
\begin{equation}\label{eq:AB}
 A_v=\underline c^{\ell_v-1}\varepsilon_-^{\ell_v},
 \qquad
 B_v=\overline C^{\ell_v-1}\varepsilon_+^{\ell_v}.
\end{equation}
Let
\[
 \delta_n=\tau\chi^{n-n_0}R_n^{-\alpha},
 \qquad
 t_{n_0}=0,
 \qquad
 t_{n+1}=t_n+\delta_n\quad(n\geq n_0).
\]
We call \(t_n\) the activation time of generation \(n\);
\eqref{eq:lower-barrier} asserts that the lower bound persists from
\(t_n\) onward.
If
\(0<T\le C_{\rm time}R_{n_0}^{-\alpha}\), then
\begin{equation}\label{eq:barrier}
 \abs{x_v(t)}\leq B_v
 \qquad(0\leq t\leq T)
\end{equation}
for every \(v\in\Gamma^{(N)}\).  In addition,
\begin{equation}\label{eq:lower-barrier}
 x_v(t)\geq A_v
\end{equation}
for every \(v\in\Gamma^{(N)}\) in the two base generations when
\(0\leq t\leq T\), and for every \(v\in\Gamma^{(N)}\) of generation
\(n>n_0\) when \(t_n\leq t\leq T\).
The choices can be made uniformly for \(1/2\leq\rho\leq1\).
\end{lemma}

\begin{proof}
\noindent\emph{Normalized finite system.}
Write \(d_v=\abs{v}^\alpha\).  For a positive mode in the finite truncation,
the normalized equation has the form
\begin{equation}\label{eq:normalized-ode}
 \dot x_v+d_v x_v
 =a_v x_p x_q
 +\sum_{v+u=z}\varsigma_{vuz}r_{vuz}x_u x_z.
\end{equation}
Here the sum runs over triads \(v+u=z\) with \(v,u,z\in\Gamma_+\), where
\(v\) is an input and \(z\) belongs to a later generation.  There is one
term for each such triad.  Terms involving a mode outside
\(\Gamma^{(N)}\) are omitted.  The generating term \(a_v x_p x_q\) is omitted
when \(v\) has no generating relation in \(\Gamma\).  For each present
generating triad \(p+q=v\), the coefficient is
\[
 a_v=\abs{\mathbf{c}_{v;p,q}}\frac{w_p w_q}{w_v}>0.
\]
For a feedback triad \(v+u=z\), the inputs in the equation at \(v\) are
\(z\) and \(-u\).  We separate the feedback coefficient into its positive
magnitude and its sign by defining
\[
 \begin{aligned}
 r_{vuz}&=\abs{\mathbf{c}_{v;z,-u}}\frac{w_u w_z}{w_v}>0,\\
 \varsigma_{vuz}
 &=\operatorname{sgn}\!\left(
 \mathbf{c}_{v;z,-u}\mathbf{c}_{z;v,u}\right)\in\{-1,1\}.
 \end{aligned}
\]
Indeed, the recursive sign choice gives
\[
 \mathbf{s}_z=\operatorname{sgn}(\mathbf{c}_{z;v,u})
 \mathbf{s}_v\mathbf{s}_u,
 \qquad
 \mathbf{s}_v\mathbf{s}_u\mathbf{s}_z
 =\operatorname{sgn}(\mathbf{c}_{z;v,u}).
\]
Together with \(\mathbf{s}_{-u}=\mathbf{s}_u\) and \(w_{-u}=w_u\), this
shows that the transformed feedback coefficient is exactly
\(\varsigma_{vuz}r_{vuz}\), as used in \eqref{eq:normalized-ode}.

Local theory for ordinary differential equations (ODEs) gives a unique
maximal solution of the finite-dimensional system.  Pairing its Fourier
equations with the conjugate coefficients
gives the exact dissipative \(L^2\) identity.  This bounds every coordinate,
so the polynomial ODE continuation criterion makes the solution global.
Moreover, there are constants
\(c_1,C_1,C_2,c_2>0\), independent of
\(n_0,N\) and \(\rho\), such that
\begin{equation}\label{eq:coefficient-bounds}
 c_1d_v\leq a_v\leq C_1d_v,
 \qquad
 \abs{r_{vuz}}\leq C_2d_ve^{-c_2R_{{\mathcal{T}_{\mathrm{gen}}}(v)}}.
\end{equation}
The lower bound \(c_1d_v\leq a_v\) is used only when the generating
term is present.  The bounds in \eqref{eq:coefficient-bounds} follow from
\eqref{eq:normalized-generating},
\eqref{eq:normalized-feedback}, and the uniform frequency comparability
within each triad.  The recurrence also shows that a mode is an input to at
most two triads whose outputs lie in later generations.

\smallskip
\noindent\emph{Barrier identities and choice of constants.}
For every nonexceptional generating triad, the degree identity gives
\[
 A_k=\underline c A_p A_q,
 \qquad
 B_k=\overline C B_p B_q.
\]
Suppose \(v+u=z\), with \(z\) in a later generation, is a triad in which
\(v\) is an input.  Under the upper barriers,
\begin{equation}\label{eq:upper-feedback-ratio}
 \frac{B_u B_z}{B_v}
 =\overline C^{-1}(\overline C\varepsilon_+)^{2\ell_u}.
\end{equation}
Let \(m_\ell>0\) be such that \(\ell_u\geq m_\ell\ell_v\) whenever
\(u,v\) are the two inputs of such a triad.  Such an \(m_\ell\), independent
of the generation, exists by
\eqref{eq:cascade-degree-growth}.

Choose the constants in the following order.  With \(\tau>0\) and
\(\chi\in(\lambda,\lambda^\alpha)\) fixed, choose
\(\overline C\geq\max\{2,4C_1\}\).  The base damping satisfies
\[
 d_vT\leq C_b
\]
on both base generations, where \(C_b\) depends only on \(\alpha\),
\(C_{\rm time}\), and the fixed mode recurrence.  Put
\(r_{\rm amp}=8e^{C_b}\).  By frequency comparability, there is a constant
\(c_*>0\), depending only on \(\alpha\) and the fixed recurrence, such that,
for an output \(k\) in generation \(n+1\), with \(j=n-n_0\),
\begin{equation}\label{eq:uniform-activation-scale}
 d_k\delta_n\geq c_*\tau\chi^j\geq c_*\tau.
\end{equation}
Since \(a_k/d_k\geq c_1\), integration over one activation interval
gives the uniform lower bound
\begin{equation}\label{eq:c3-definition}
 \int_0^{\delta_n}e^{-d_k(\delta_n-s)}a_k\,ds
 =\frac{a_k}{d_k}(1-e^{-d_k\delta_n})
 \geq c_1(1-e^{-c_*\tau})=:c_3>0.
\end{equation}
Fix \(0<\eta<1/4\), and then choose
\(\underline c\in(0,1)\) so small that
\begin{equation}\label{eq:lower-constant-choice}
 \frac{c_1}{\underline c}>2(1+\eta),
 \qquad
 \frac{c_3}{\underline c}\geq8.
\end{equation}
The parameter \(\varepsilon_0\) will be chosen small below.  Put
\begin{equation}\label{eq:epsilon-choice}
 \varepsilon_+=2\varepsilon_0,
 \qquad \varepsilon_-=\frac{2\varepsilon_0}{r_{\rm amp}}.
\end{equation}

\smallskip
\noindent\emph{Upper barrier.}
Set \(q_*=\overline C\varepsilon_+\).  Taking \(q_*\) sufficiently small in
\eqref{eq:upper-feedback-ratio} gives
\begin{equation}\label{eq:upper-feedback-bound}
 \sum_{v+u=z}\abs{r_{vuz}}B_u B_z
 \leq\eta d_v B_v.
\end{equation}
For a feedback interaction belonging to
\eqref{eq:exceptional-base-triads}, all three homogeneity degrees equal one,
and the ratio becomes
\begin{equation}\label{eq:exceptional-upper-ratio}
 \frac{B_u B_z}{B_v}=\varepsilon_+.
\end{equation}
There is only one such triad, and a further reduction of
\(\varepsilon_+\) makes its contribution satisfy
\eqref{eq:upper-feedback-bound}.
At a hypothetical first contact \(\abs{x_v}=B_v\), the upper right Dini
derivative satisfies, for a non-base mode,
\[
 D^+\abs{x_v}
 \leq d_v B_v\left(-1+\frac{C_1}{\overline C}+\eta\right)<0.
\]
For a base mode, the possible generating term is bounded instead by
\(C_1d_v\varepsilon_+B_v\), which gives the same conclusion after reducing
\(\varepsilon_+\).  Since initially
\(\varepsilon_0<B_v\) on the base and all other modes vanish, this proves
\eqref{eq:barrier} on \([0,T]\), uniformly in \(N\).

\smallskip
\noindent\emph{Feedback control for the lower barrier.}
For the lower barrier, bound each potentially negative feedback term by the
two upper barriers.  For a nonexceptional feedback triad, the resulting
ratio is
\begin{equation}\label{eq:lower-feedback-ratio}
 \frac{B_u B_z}{A_v}
 =\overline C^{2\ell_u-1}
 \left(\frac{\overline C}{\underline c}\right)^{\ell_v-1}
 \varepsilon_+^{2\ell_u}
 \left(\frac{\varepsilon_+}{\varepsilon_-}\right)^{\ell_v}.
\end{equation}
Since \(\varepsilon_+/\varepsilon_-=r_{\rm amp}\) and
\(\ell_u\geq m_\ell\ell_v\), the right-hand side of
\eqref{eq:lower-feedback-ratio} is bounded by a fixed factor times
\[
 \left[q_*^{2m_\ell}
       \frac{\overline C r_{\rm amp}}{\underline c}\right]^{\ell_v}.
\]
The factor in brackets can be made arbitrarily small by reducing \(q_*\).
Together with \eqref{eq:coefficient-bounds} and the uniform bound on the
number of triads with outputs in later generations, this gives
\begin{equation}\label{eq:feedback-error-bound}
 \sum_{v+u=z}
 \abs{r_{vuz}x_u x_z}
 \leq\eta d_v A_v.
\end{equation}
For an exceptional base triad the corresponding ratio is instead
\begin{equation}\label{eq:exceptional-lower-ratio}
 \frac{B_u B_z}{A_v}
 =\frac{\varepsilon_+^2}{\varepsilon_-}
 =r_{\rm amp}\varepsilon_+.
\end{equation}
After a further reduction of \(\varepsilon_+\), the regular and exceptional
feedback contributions together satisfy \eqref{eq:feedback-error-bound}.
In particular, for every \(0\leq t\leq T\), the already established upper
barrier gives
\[
 \left|\sum_{v+u=z}\varsigma_{vuz}r_{vuz}x_u(t)x_z(t)\right|
 \leq\sum_{v+u=z}\abs{r_{vuz}}B_uB_z
 \leq\eta d_vA_v.
\]
Here \(A_v\) is a fixed prescribed level; no lower bound on \(x_v\),
\(x_u\), or \(x_z\), and no assumption on their signs, is used.
Thus the feedback estimate remains valid before activation, including
before the lower barrier at \(v\) has been established.

\smallskip
\noindent\emph{Base generations.}
Generation \(n_0-1\) has no generating term.  Duhamel's
formula and \eqref{eq:feedback-error-bound} give
\[
 x_v(t)
 \geq e^{-d_vt}\varepsilon_0-\eta A_v
 \geq e^{-C_b}\frac{\varepsilon_+}{2}-\eta\varepsilon_-
 =(4-\eta)A_v>A_v.
\]
Both inputs of \(U_{n_0}\) lie in generation \(n_0-1\) and already satisfy
their lower barriers, so its generating term is nonnegative.  The mode
\(V_{n_0}\) has no generating term because its recurrence would require
\(V_{n_0-2}\), which lies outside \(\Gamma\).  Thus the same estimate proves
the lower barrier on generation \(n_0\).

\smallskip
\noindent\emph{Finite initialization range.}
We first treat finitely many generations without assuming that their
generating products have a favorable sign before activation.  By
\eqref{eq:AB} and \eqref{eq:epsilon-choice},
\[
 \frac{B_v}{A_v}
 =\left(\frac{\overline C}{\underline c}\right)^{\ell_v-1}
  \left(\frac{\varepsilon_+}{\varepsilon_-}\right)^{\ell_v}
 =\left(\frac{\overline C}{\underline c}\right)^{\ell_v-1}
  r_{\rm amp}^{\ell_v}.
\]
In particular, this ratio is independent of \(\varepsilon_0\).
Put \(j=n-n_0\).  From \eqref{eq:cascade-degree-growth},
\begin{equation}\label{eq:BA-growth}
 \log\frac{B_k}{A_k}\leq C_{\rm deg}\lambda^j
\end{equation}
for an output \(k\) in generation \(n+1\), where \(C_{\rm deg}\) depends
only on the fixed comparison parameters and is independent of
\(\varepsilon_0\).  On the other hand,
\begin{equation}\label{eq:activation-damping}
 d_k\delta_n\geq c_*\tau\chi^j.
\end{equation}
Since \(\chi>\lambda\), there is an integer \(J\geq1\), independent of
\(n_0\), \(N\), \(\rho\), and \(\varepsilon_0\), such that
\begin{equation}\label{eq:old-value-damped}
 e^{-d_k\delta_n}B_k\leq A_k
 \qquad(j\geq J).
\end{equation}
Thus \(J\) can be fixed before choosing \(\varepsilon_0\); subsequent
reductions of \(\varepsilon_0\) do not alter this choice.

To apply Lemma~\ref{lem:finite-depth-initialization} on this range, set
\[
 T_J=\min\{T,t_{n_0+J}\},
 \qquad s=R_{n_0}^\alpha t,
 \qquad
 y_v(s)=\varepsilon_0^{-\ell_v}
 x_v(R_{n_0}^{-\alpha}s),
\]
and take as the finite coordinate set
\begin{equation}\label{eq:initialization-coordinate-set}
 \mathcal I=\bigcup_{r=-1}^{J}\mathcal I_r,
 \qquad
 \mathcal I_r
 =\left\{v\in\Gamma_+\cap\Gamma^{(N)}:
   {\mathcal T}_{\rm gen}(v)=n_0+r\right\}.
\end{equation}
(Thus \(\mathcal I_r=\varnothing\) when \(n_0+r>N\).)
The generating inputs of a regular coordinate in \(\mathcal I_r\) lie
in earlier sets \(\mathcal I_{r'}\), as required by the abstract lemma.
Coordinates outside \(\mathcal I\) may occur only through the perturbation
term below.
Introduce the rescaled coefficients
\[
 D_v=\frac{d_v}{R_{n_0}^\alpha},
 \qquad \widetilde a_v=\frac{a_v}{R_{n_0}^\alpha},
 \qquad \widetilde r_{vuz}=\frac{r_{vuz}}{R_{n_0}^\alpha}.
\]
For every target coordinate \(v\in\mathcal I\), the exact rescaled equation
is
\begin{equation}\label{eq:initialization-rescaled}
 \frac{dy_v}{ds}+D_v y_v
 =\mathbf 1_{\mathrm{reg}}(v)\widetilde a_v y_p y_q
  +\mathcal E_v^{\varepsilon_0}(y),
\end{equation}
where \(\mathbf 1_{\mathrm{reg}}(v)\) indicates that \(v\) has a present
nonexceptional generating triad.  The remaining terms are collected in
\begin{equation}\label{eq:initialization-remainder}
 \begin{split}
 \mathcal E_v^{\varepsilon_0}(y)
 ={}&\mathbf 1_{\mathrm{exc}}(v)
 \varepsilon_0^{\ell_p+\ell_q-\ell_v}
 \widetilde a_v y_p y_q+\sum_{v+u=z}\varsigma_{vuz}
 \varepsilon_0^{\ell_u+\ell_z-\ell_v}
 \widetilde r_{vuz}y_u y_z.
 \end{split}
\end{equation}
Here \(\mathbf 1_{\mathrm{exc}}(v)\) indicates that the exceptional base
triad at \(v\) is present in the truncation.  Terms involving a mode outside the truncation are
absent.  For a nonexceptional feedback interaction,
\(\ell_z=\ell_v+\ell_u\), and hence
\[
 \ell_u+\ell_z-\ell_v=2\ell_u\geq2.
\]
For the exceptional base triad, all three degrees are one.  Thus both the
generating exponent \(\ell_p+\ell_q-\ell_v\) and each associated feedback
exponent \(\ell_u+\ell_z-\ell_v\) equal one.  Consequently, every monomial in
\(\mathcal E^{\varepsilon_0}\) carries at least one explicit factor of
\(\varepsilon_0\).

By the bounded generation gaps in the triad relations, all modes occurring
on the right-hand side of an equation of generation at most \(n_0+J\) have
generation at most \(n_0+J+2\).  On these finitely many relative
generations, \eqref{eq:coefficient-bounds} and frequency comparability give
\begin{equation}\label{eq:initialization-uniform-coefficients}
 0\leq D_v\leq K_J,
 \qquad 0\leq\widetilde a_v\leq K_J,
 \qquad \abs{\widetilde r_{vuz}}\leq K_J,
\end{equation}
with \(K_J\) independent of \(n_0\), \(N\), and \(\rho\).  Moreover, the
upper barrier and \(\varepsilon_+=2\varepsilon_0\) give
\begin{equation}\label{eq:initialization-rescaled-upper}
 \abs{y_u(s)}
 \leq\frac{B_u}{\varepsilon_0^{\ell_u}}
 =\overline C^{\ell_u-1}2^{\ell_u}
 \leq M_J
 \qquad(0\leq s\leq R_{n_0}^\alpha T_J,
        \ {\mathcal{T}_{\mathrm{gen}}}(u)\leq n_0+J+2).
\end{equation}
Here \(M_J\) is uniform because \eqref{eq:cascade-degree-growth} bounds
\(\ell_u\) at every fixed relative depth.  The uniform bound on the number
of triads containing a mode, together with
\eqref{eq:initialization-remainder} and
\eqref{eq:initialization-rescaled-upper}, gives
\begin{equation}\label{eq:initialization-remainder-bound}
 \max_{v\in\mathcal I}
 \sup_{0\leq s\leq R_{n_0}^\alpha T_J}
 \abs{\mathcal E_v^{\varepsilon_0}(y(s))}
 \leq K'_J\varepsilon_0.
\end{equation}

Let \(y^{(0)}\) solve the triangular system obtained from
\eqref{eq:initialization-rescaled} by setting
\(\mathcal E^{\varepsilon_0}=0\).  It has the same initial data as \(y\):
value one on the two base generations and zero on higher generations.  The
rescaled time interval has length at most \(C_{\rm time}\), since
\[
 R_{n_0}^\alpha T_J\leq R_{n_0}^\alpha T\leq C_{\rm time}.
\]
We may therefore apply Lemma~\ref{lem:finite-depth-initialization} with
\[
 Y^\varepsilon=y,
 \qquad Y^0=y^{(0)},
 \qquad \varepsilon=\varepsilon_0,
 \qquad S=R_{n_0}^\alpha T_J,
\]
and depth \(J\); the sets beyond the truncation are simply empty.
Since \(S\leq C_{\rm time}\), the recursive bound for the triangular
solution and the explicit Gronwall constant in the proof of the lemma
can both be bounded using \(C_{\rm time}\).  The resulting constants
are therefore uniform in \(n_0\), \(N\), and \(\rho\).  The bounds in
\eqref{eq:initialization-uniform-coefficients},
\eqref{eq:initialization-rescaled-upper}, and
\eqref{eq:initialization-remainder-bound} supply uniform choices of
\(K_0\) and \(M\).  The perturbation parameter in
Lemma~\ref{lem:finite-depth-initialization} is therefore the base amplitude
\(\varepsilon_0\), not an additional parameter of the SQG equation.

For comparison in physical time, define
\[
 \widetilde y_v(t)=y_v(R_{n_0}^\alpha t),
 \qquad
 \widetilde y_v^{(0)}(t)=y_v^{(0)}(R_{n_0}^\alpha t).
\]
After returning to physical time,
Lemma~\ref{lem:finite-depth-initialization} yields the uniform triangular bound
\begin{equation}\label{eq:initialization-triangular-bound}
 \max_{v\in\mathcal I}
 \sup_{0\leq t\leq T_J}\abs{\widetilde y_v^{(0)}(t)}\leq M_J^{(0)}
\end{equation}
and the approximation
\begin{equation}\label{eq:initialization-convergence}
 \max_{v\in\mathcal I}
 \sup_{0\leq t\leq T_J}
 \abs{\widetilde y_v(t)-\widetilde y_v^{(0)}(t)}
 \leq C_J\varepsilon_0.
\end{equation}
These constants are independent of \(n_0\), \(N\), and \(\rho\).  If a mode
required only by a feedback interaction lies outside \(\Gamma^{(N)}\), the
corresponding term is absent from \(\mathcal E^{\varepsilon_0}\).

It remains to extract a lower bound from the triangular reference system.
Normalize the lower barrier by setting
\(\overline A_v=A_v/\varepsilon_0^{\ell_v}\).  Then
\[
 \begin{aligned}
 \overline A_k&=\underline c\,\overline A_p\overline A_q
 &&\text{on every nonexceptional generating triad},\\
 \overline A_v&=\frac2{r_{\rm amp}}
 &&\text{on the base generations}.
 \end{aligned}
\]
Therefore
\begin{equation}\label{eq:aJ-positive}
 a_J:=\min_{\substack{v\in\Gamma_+\\
                       n_0-1\leq {\mathcal{T}_{\mathrm{gen}}}(v)\leq n_0+J}}
       \overline A_v>0
\end{equation}
depends only on \(J\), \(\underline c\), and \(r_{\rm amp}\), and is
independent of \(n_0\), \(N\), \(\rho\), and \(\varepsilon_0\).  The
variation-of-constants formula shows that the triangular solution is
nonnegative.  On a base mode it is exactly
\(e^{-d_vt}\), and hence, for \(0\leq t\leq T_J\),
\[
 \widetilde y_v^{(0)}(t)\geq e^{-C_b}=4\overline A_v.
\]
Suppose inductively that the two input modes of a non-base mode \(k\) obey
\(\widetilde y_p^{(0)}\geq\overline A_p\) and
\(\widetilde y_q^{(0)}\geq\overline A_q\) on its activation interval.  If
\(t_{n+1}\leq T_J\), the same Duhamel integral used in the definition of
\(c_3\) gives
\[
 \widetilde y_k^{(0)}(t_{n+1})
 \geq c_3\overline A_p\overline A_q
 =\frac{c_3}{\underline c}\overline A_k
 \geq8\overline A_k.
\]
The bound then persists after activation.  Indeed, at a hypothetical first
downward contact with \(\overline A_k\),
\[
 \frac{d\widetilde y_k^{(0)}}{dt}
 \geq d_k\overline A_k
 \left(-1+\frac{c_1}{\underline c}\right)>0.
\]
Induction over the finitely many initialization generations therefore gives
\(\widetilde y_k^{(0)}(t_{n+1})\geq8\overline A_k\) at every non-base
activation time not exceeding \(T_J\), and preserves the lower bound
\(\overline A_k\) thereafter.

Reduce \(\varepsilon_0\), without changing the constants already fixed, until
\begin{equation}\label{eq:initialization-epsilon-choice}
 C_J\varepsilon_0<5a_J.
\end{equation}
For every non-base mode in the initialization range whose activation time is
at most \(T\), \eqref{eq:initialization-convergence} yields
\begin{equation}\label{eq:initialization-lower}
 x_v(t_n)=\varepsilon_0^{\ell_v}\widetilde y_v(t_n)\geq3A_v
 \qquad(n_0<n\leq n_0+J,\quad t_n\leq T).
\end{equation}
For these finitely many initialization generations, the lower bound
persists after activation.  Indeed, argue inductively in the generation
and let \(t_*>t_n\) be a hypothetical first downward contact at which
\(x_v(t_*)=A_v\).  The lower bounds for the input modes and
\eqref{eq:feedback-error-bound} give
\[
 \dot x_v(t_*)
 \geq d_vA_v
 \left(-1-\eta+\frac{c_1}{\underline c}\right)>0,
\]
contradicting the definition of \(t_*\).  Thus \(x_v(t)\geq A_v\) for
\(t_n\leq t\leq T\).  This is the only place where the finite-depth
perturbation argument is used.

\smallskip
\noindent\emph{Later generations.}
We now continue the induction, considering only activation times not
exceeding \(T\).  The finite initialization range is covered by
\eqref{eq:initialization-lower}.  Let \(k\) belong to generation \(n+1\),
where \(n-n_0\geq J\), and suppose its input modes satisfy their lower
barriers on
\([t_n,t_{n+1}]\).  Since \(d_k{\approx} R_n^\alpha\) and
\(d_k\delta_n\geq c_*\tau\), the generating part of Duhamel's formula gives
\[
 \int_{t_n}^{t_{n+1}}
 e^{-d_k(t_{n+1}-s)}a_k A_p A_q\,ds
 \geq c_3 A_p A_q.
\]
The absolute upper barrier, \eqref{eq:old-value-damped}, and
\eqref{eq:feedback-error-bound} therefore give
\[
 x_k(t_{n+1})
 \geq-e^{-d_k\delta_n}B_k+c_3 A_p A_q-\eta A_k
 \geq\left(\frac{c_3}{\underline c}-1-\eta\right)A_k
 \geq3A_k
\]
by \eqref{eq:lower-constant-choice}.  This initializes the lower bound at
generation \(n+1\).  It also persists: at a hypothetical first downward
contact of \(x_k\) with \(A_k\), the input-mode lower bounds and
\eqref{eq:feedback-error-bound} give
\[
 \dot x_k
 \geq d_k A_k
 \left(-1-\eta+\frac{c_1}{\underline c}\right)>0,
\]
again by \eqref{eq:lower-constant-choice}, a contradiction.  Iterating the
initialization and persistence steps proves \eqref{eq:lower-barrier}.  All
choices are independent of \(N\) and \(n_0\), and uniform for
\(1/2\leq\rho\leq1\).
\end{proof}

\subsection{Global solution and proof of Theorem~\ref{thm:SQG-main}}

With the sparse geometry and uniform cascade estimates in place, we now
pass from the Galerkin systems to the full projected equation.  We first
establish global existence and uniqueness for arbitrary smooth initial data
whose Fourier support is contained in \(\Gamma\).
Proposition~\ref{prop:global} will also justify passing the comparison bounds
from finite truncations to the infinite cascade.

\begin{proposition}[Global existence and uniqueness on the sparse Fourier set]
\label{prop:global}
Let \(1<\alpha<2\).  For every real-valued \(C^\infty\) datum on
\(\T^2\) whose Fourier support is contained in \(\Gamma\), equation
\eqref{eq:projected-SQG} has a unique global real-valued smooth solution.
This solution satisfies the exact energy identity
\[
 \frac12\norm{\theta(t)}_2^2
 +\int_0^t\norm{\Lambda^{\alpha/2}\theta(s)}_2^2\,ds
 =\frac12\norm{\theta(0)}_2^2
 \qquad(t\geq0).
\]
\end{proposition}

\begin{proof}
Write \(\theta^{\rm in}\) for the initial datum.

\smallskip
\noindent\emph{Galerkin energy identity.}
Lemma~\ref{lem:geometry} and additive separation imply the finite-interaction
estimate
\[
 \abs{\widehat{\mathcal Q_\Gamma(\theta,\theta)}(k)}
 \leq C\abs{k}^\gamma
 \sum_{(p,q)\in\mathcal N(k)}
 \abs{\widehat\theta(p)\widehat\theta(q)},
 \qquad \#\mathcal N(k)\leq C,
\]
where \(\mathcal N(k)\) is the set of ordered input pairs that contribute
to the output \(k\).  Moreover, the three frequencies in each contributing
triad are comparable, with constants independent of the generation.

Let \(\theta_N\) solve the Galerkin system on \(\Gamma^{(N)}\), with
\[
 \theta_N(0)=\Pi_{\Gamma^{(N)}}\theta^{\rm in}.
\]
For every Galerkin state \(\theta\), the vector field
\(\nabla^\perp\Lambda^{-1}\theta\) is divergence free.  Since the Fourier
projection is orthogonal and \(\theta\) belongs to its range,
\[
 \left\langle
 \mathcal Q_{\Gamma^{(N)}}(\theta,\theta),\theta
 \right\rangle
 =-\int_{\T^2}
 (\nabla^\perp\Lambda^{-1}\theta\cdot\nabla\theta)\theta\,dx
 =0.
\]
The boundary term vanishes by periodicity.  Thus every finite truncation
satisfies the exact \(L^2\) identity.  In particular, this identity prevents
finite-time blow-up of the Galerkin coordinates, so the local solution of
the polynomial ODE system extends globally.

\smallskip
\noindent\emph{Uniform Sobolev estimates.}
For \(s\geq0\), let \(\mathfrak T\) denote the set of unordered
contributing triads and let \(R_{\mathfrak t}\) denote the common magnitude
scale of the three frequencies in \(\mathfrak t\).  Frequency comparability
shows that the absolute value of
the contribution of \(\mathfrak t=\{p,q,k\}\) is bounded by
\(CR_{\mathfrak t}^{2s+\gamma}
  \abs{\widehat\theta(p)\widehat\theta(q)\widehat\theta(k)}\).
After bounding one factor in \(\ell^\infty\), Cauchy--Schwarz and the
uniform bound on the number of triads containing a given mode give
\[
 \begin{aligned}
 &\sum_{\mathfrak t\in\mathfrak T}
 R_{\mathfrak t}^{2s+\gamma}
 \abs{\widehat\theta(p)\widehat\theta(q)\widehat\theta(k)}\\
 &\quad\leq
 \norm{\widehat\theta}_{\ell^\infty}
 \sum_{\mathfrak t\in\mathfrak T}
 \bigl(R_{\mathfrak t}^{s+\gamma/2}\abs{\widehat\theta(p)}\bigr)
 \bigl(R_{\mathfrak t}^{s+\gamma/2}\abs{\widehat\theta(k)}\bigr)\\
 &\quad\leq C\norm{\theta}_2
 \norm{\Lambda^{s+\gamma/2}\theta}_2^2.
 \end{aligned}
\]
Here the choice of which factor is placed in \(\ell^\infty\) and the
permutations within a triad are absorbed into the constant.  Consequently,
\[
 \left|\left\langle
 \Lambda^s\mathcal Q_\Gamma(\theta,\theta),
 \Lambda^s\theta
 \right\rangle\right|
 \leq C\norm{\theta}_2
 \norm{\Lambda^{s+\gamma/2}\theta}_2^2.
\]
The same estimate holds uniformly for \(\Gamma^{(N)}\).  Apply this to
\(\theta_N\).  With \(\vartheta=\gamma/\alpha\in(0,1)\), interpolation
gives
\[
 \norm{\Lambda^{s+\gamma/2}\theta_N}_2^2
 \leq
 \norm{\Lambda^s\theta_N}_2^{2(1-\vartheta)}
 \norm{\Lambda^{s+\alpha/2}\theta_N}_2^{2\vartheta}.
\]
Young's inequality and the \(L^2\) identity therefore yield
\[
 \frac d{dt}\norm{\Lambda^s\theta_N}_2^2
 +\norm{\Lambda^{s+\alpha/2}\theta_N}_2^2
 \leq
 C_{\alpha,s}\norm{\theta^{\rm in}}_2^{\alpha/(\alpha-\gamma)}
 \norm{\Lambda^s\theta_N}_2^2,
\]
uniformly in \(N\).  Gronwall's inequality then gives, for every finite
\(S\) and every \(s\geq0\),
\begin{equation}\label{eq:uniform-Hs}
 \sup_{0\leq t\leq S}\norm{\theta_N(t)}_{H^s}^2
 +\int_0^S\norm{\theta_N(t)}_{H^{s+\alpha/2}}^2\,dt
 \leq C_{S,s,\theta^{\rm in}}.
\end{equation}
Here the homogeneous and inhomogeneous Sobolev norms are equivalent on
\(\operatorname{Ran}\Pi_\Gamma\), because \(0\notin\Gamma\).

\smallskip
\noindent\emph{Bilinear estimate and compactness.}
Define \(\mathcal B_\Gamma(f,g)\) by replacing
\(\widehat f(p)\widehat f(q)\) on the right-hand side of
\eqref{eq:SQG-symbol} with \(\widehat f(p)\widehat g(q)\).  The multiplier
in \eqref{eq:SQG-symbol} is symmetric under \(p\leftrightarrow q\), so
\(\mathcal B_\Gamma(f,g)=\mathcal B_\Gamma(g,f)\) and
\(\mathcal Q_\Gamma(f,f)=\mathcal B_\Gamma(f,f)\).  For \(s\geq0\) and
functions \(f,g\) whose Fourier supports are contained in \(\Gamma\),
bounded interaction degree gives the polarized estimate
\begin{equation}\label{eq:polarized-Q-map}
 \norm{\mathcal B_\Gamma(f,g)}_{H^{s-\gamma}}
 \leq C\bigl(\norm{f}_2\norm{g}_{H^s}
              +\norm{g}_2\norm{f}_{H^s}\bigr).
\end{equation}
Taking \(g=f\) gives
\begin{equation}\label{eq:Q-map}
 \norm{\mathcal Q_\Gamma(f,f)}_{H^{s-\gamma}}
 \leq C\norm{f}_2\norm{f}_{H^s}.
\end{equation}
Indeed, frequency comparability and the multiplier bound imply
\[
 \abs{k}^{s-\gamma}
 \abs{\widehat{\mathcal B_\Gamma(f,g)}(k)}
 \leq C\sum_{(p,q)\in\mathcal N(k)}
 \abs{\widehat f(p)}\,\abs{q}^s\abs{\widehat g(q)}.
\]
Squaring, summing in \(k\), and using the uniformly bounded interaction
degree give
\[
 \norm{\mathcal B_\Gamma(f,g)}_{H^{s-\gamma}}^2
 \leq C\norm{\widehat f}_{\ell^\infty}^2
          \norm{g}_{H^s}^2
 \leq C\norm f_2^2\norm g_{H^s}^2.
\]
Interchanging \(f\) and \(g\) gives \eqref{eq:polarized-Q-map}.

Because \(\gamma<\alpha/2\), estimate \eqref{eq:Q-map},
\eqref{eq:uniform-Hs}, and the equation imply
\[
 \sup_N\norm{\partial_t\theta_N}_{L^2(0,S;H^{s-\alpha/2})}<\infty.
\]

Extend each Galerkin solution by zero outside \(\Gamma^{(N)}\), thereby
regarding all \(\theta_N\) as functions on the same torus.  The compact
embedding
\(H^{s+\alpha/2}(\T^2)\Subset H^s(\T^2)\), the displayed time-derivative
bound, and the Aubin--Lions lemma \cite{Simon1987} give, after a diagonal
extraction over positive integer values of \(s,S\),
\begin{equation}\label{eq:strong-Galerkin-convergence}
 \theta_N\longrightarrow\theta
 \quad\text{strongly in }L^2(0,S;H^s)
 \quad\text{for every }s\geq0.
\end{equation}
Convergence at an arbitrary real index \(s\geq0\) follows by carrying out
the extraction at a larger integer index.  To pass to the nonlinear term,
write
\[
 \mathcal Q_\Gamma(\theta_N,\theta_N)
 -\mathcal Q_\Gamma(\theta,\theta)
 =\mathcal B_\Gamma(\theta_N-\theta,\theta_N)
  +\mathcal B_\Gamma(\theta,\theta_N-\theta).
\]
Applying \eqref{eq:polarized-Q-map} with input regularity
\(s+\gamma\), together with \eqref{eq:uniform-Hs} and
\eqref{eq:strong-Galerkin-convergence}, shows that the left-hand side
converges strongly to zero in \(L^1(0,S;H^s)\) for every \(s\geq0\).

To account for the finite output projection, put
\(P_N=\Pi_{\Gamma^{(N)}}\).  Since the Fourier support of \(\theta_N\) is
contained in \(\Gamma^{(N)}\),
\[
 \mathcal Q_{\Gamma^{(N)}}(\theta_N,\theta_N)
 =P_N\mathcal Q_\Gamma(\theta_N,\theta_N).
\]
Moreover,
\[
 \begin{aligned}
 &P_N\mathcal Q_\Gamma(\theta_N,\theta_N)
       -\mathcal Q_\Gamma(\theta,\theta)\\
 &\quad=P_N\bigl[\mathcal Q_\Gamma(\theta_N,\theta_N)
                    -\mathcal Q_\Gamma(\theta,\theta)\bigr]
       +(P_N-\mathrm{Id})\mathcal Q_\Gamma(\theta,\theta).
 \end{aligned}
\]
The first term tends to zero by the strong \(L^1(0,S;H^s)\) convergence
established after \eqref{eq:strong-Galerkin-convergence}, and the second tends
to zero because \(P_N\to\mathrm{Id}\) strongly on every
Sobolev space within \(\operatorname{Ran}\Pi_\Gamma\).  Hence the finite
nonlinearities converge strongly in \(L^1(0,S;H^s)\), and the limit
satisfies \eqref{eq:projected-SQG} in the distributional sense.

\smallskip
\noindent\emph{Initial value and smoothness.}
For each fixed Fourier mode, the Galerkin equations and the uniform bounds
give a uniform \(W^{1,2}(0,S)\) estimate for the corresponding coefficient.
A further countable diagonal extraction therefore gives uniform convergence
of every Fourier coefficient on \([0,S]\).  In particular, the limit attains
the prescribed initial value.  For the Galerkin solutions used in
Lemma~\ref{lem:barrier}, this coordinatewise convergence also transfers the
upper and lower barriers to the limit.

Weak compactness in \eqref{eq:uniform-Hs} and the equation give, for every
\(s\geq0\),
\[
 \theta\in L^2(0,S;H^{s+\alpha/2}),
 \qquad
 \partial_t\theta\in L^2(0,S;H^{s-\alpha/2}).
\]
The Lions--Magenes theorem \cite{LionsMagenes1972} for the Hilbert triple
\[
 H^{s+\alpha/2}\subset H^s\subset H^{s-\alpha/2}
\]
yields \(\theta\in C([0,S];H^s)\).  Since this holds for every \(s\), the
solution is smooth in space.  More explicitly, for any \(r\geq0\),
\(\theta\in C([0,S];H^{r+\alpha})\) implies that
\(\Lambda^\alpha\theta\in C([0,S];H^r)\), while
\eqref{eq:polarized-Q-map}, applied with \(s=r+\gamma\), gives
\(\mathcal B_\Gamma(\theta,\theta)\in C([0,S];H^r)\).
The equation then gives
\(\partial_t\theta\in C([0,S];H^r)\).  Differentiating the equation and
using bilinearity, the same estimate yields inductively
\(\partial_t^j\theta\in C([0,S];H^r)\) for every
\(j\in\mathbb N_0\) and \(r\geq0\).
Thus \(\theta\) is smooth in space and time.  As \(S<\infty\) was
arbitrary, the solution is global.

\smallskip
\noindent\emph{Uniqueness and the energy identity.}
Let \(\theta\) and \(\widetilde\theta\) be two smooth solutions with the
same initial datum, and set \(h=\theta-\widetilde\theta\).  Then
\begin{equation}\label{eq:uniqueness-difference-equation}
 \partial_t h+\Lambda^\alpha h
 =\mathcal B_\Gamma(h,\theta)
  +\mathcal B_\Gamma(\widetilde\theta,h),
 \qquad h(0)=0.
\end{equation}
For every \(s\geq0\), the same unordered-triad argument, applied to the
difference equation, gives
\begin{equation}\label{eq:bilinear-difference-estimate}
 \begin{split}
 &\left|\left\langle
 \Lambda^s\bigl[\mathcal B_\Gamma(h,\theta)
                  +\mathcal B_\Gamma(\widetilde\theta,h)\bigr],
 \Lambda^s h\right\rangle\right|\\
 &\hspace{2cm}\leq
 C\bigl(\norm{\theta}_2+\norm{\widetilde\theta}_2\bigr)
 \norm{\Lambda^{s+\gamma/2}h}_2^2.
 \end{split}
\end{equation}
Here one non-difference Fourier factor is bounded in \(\ell^\infty\) by
the corresponding \(L^2\) norm.  Frequency comparability, the multiplier
bound \(O(\abs{k}^\gamma)\), and bounded interaction degree then control the
two remaining \(h\)-factors by the displayed weighted \(\ell^2\) norm.

Put \(\vartheta=\gamma/\alpha\) and
\[
 M(t)=\norm{\theta(t)}_2+\norm{\widetilde\theta(t)}_2.
\]
Taking the \(H^s\) energy of
\eqref{eq:uniqueness-difference-equation}, using
\eqref{eq:bilinear-difference-estimate} and interpolation gives
\[
 \frac12\frac d{dt}\norm{\Lambda^s h}_2^2
 +\norm{\Lambda^{s+\alpha/2}h}_2^2
 \leq CM(t)
 \norm{\Lambda^s h}_2^{2(1-\vartheta)}
 \norm{\Lambda^{s+\alpha/2}h}_2^{2\vartheta}.
\]
Young's inequality with exponents \(1/\vartheta\) and
\(1/(1-\vartheta)\) therefore yields
\begin{equation}\label{eq:uniqueness-gronwall}
 \frac d{dt}\norm{\Lambda^s h}_2^2
 +\norm{\Lambda^{s+\alpha/2}h}_2^2
 \leq C_{\alpha,s}
 M(t)^{\alpha/(\alpha-\gamma)}
 \norm{\Lambda^s h}_2^2.
\end{equation}
The coefficient on the right hand side belongs to \(L^1(0,S)\) because
\(\theta,\widetilde\theta\in C([0,S];L^2)\).  Since \(h(0)=0\), Gronwall's
inequality gives \(h\equiv0\).

Finally, real-valuedness and the inclusion
\(\operatorname{supp}\widehat\theta(t)\subset\Gamma\) pass to the limit
coefficientwise.  This inclusion also implies zero mean because
\(0\notin\Gamma\).  The smooth limit therefore admits the same periodic
integration by parts as the Galerkin solutions.  The cancellation proves
the asserted energy identity.
\end{proof}

\begin{remark}[Range of the dissipation exponent]\label{rem:alpha-range}
The restriction \(\alpha>1\) enters at two points.  First, it makes the
interval \((\lambda,\lambda^\alpha)\) nonempty, as required for the choice
of \(\chi\) in Lemma~\ref{lem:barrier}.  Second, since
\eqref{eq:gamma} gives \(\gamma<1/2\), it ensures
\(\gamma<\alpha/2\), as required for the time-derivative bound used in the
Aubin--Lions argument.  Thus
these parts of the construction apply for every \(\alpha>1\).  We retain
\(1<\alpha<2\) in Theorem~\ref{thm:SQG-main} to remain within the
subcritical dissipative SQG regime considered here.
\end{remark}

\begin{proof}[Proof of Theorem~\ref{thm:SQG-main}]
We choose the cascade scale to match the prescribed time and then use the
two barriers to bound the analyticity radius from above and below.

\smallskip
\noindent\emph{Choice of the base generation.}
Fix \(\chi\in(\lambda,\lambda^\alpha)\) and \(\tau>0\), and set
\[
 C_{\rm time}=\frac{\tau\lambda^\alpha}{1-\chi\lambda^{-\alpha}}.
\]
Choose the constants in Lemma~\ref{lem:barrier} for this
\(C_{\rm time}\).  If the
two base generations are \(m-1\) and \(m\), the activation times accumulate
at
\[
 T_m
 =\sum_{n=m}^\infty
 \tau\chi^{n-m}\lambda^{-\alpha n}
 =\frac{\tau\lambda^{-\alpha m}}
 {1-\chi\lambda^{-\alpha}}.
\]
In particular,
\begin{equation}\label{eq:activation-time-scaling}
 T_{m-1}=\lambda^\alpha T_m.
\end{equation}
The degree vector obeys the same matrix recurrence as the modes.
For any choice of base generation \(m\), a generation-\(n\) mode satisfies
\[
 \ell_v\leq C\lambda^{n-m}
 \leq C\lambda^{-m}\abs{v}.
\]
Since \(0<\underline c,\varepsilon_-<1\), \eqref{eq:AB} gives
\[
 -\log A_v
 =-(\ell_v-1)\log\underline c-\ell_v\log\varepsilon_-
 \leq C\ell_v
 \leq C_A\lambda^{-m}\abs{v},
\]
where \(C_A>0\) is independent of \(m\).
Choose an integer \(N_1\geq N_*\) such that
\(C_A\lambda^{-N_1}<1/2\), and define
\[
 T_\alpha:=T_{N_1}.
\]
Given \(0<T<T_\alpha\), choose the unique integer \(n_0\) such that
\[
 T_{n_0}\leq T<\lambda^\alpha T_{n_0}.
\]
Indeed, by \eqref{eq:activation-time-scaling}, the half-open intervals
\([T_m,T_{m-1})\), \(m>N_1\), partition \((0,T_{N_1})\).  Hence this
choice is unique and satisfies \(n_0>N_1\), so
\(n_0-1\geq N_*\).  Put
\[
 \varepsilon_{\rm rad}=C_A\lambda^{-n_0},
 \qquad
 \rho=1-\varepsilon_{\rm rad}.
\]
Then \(1/2<\rho<1\), so the uniform version of
Lemma~\ref{lem:barrier} applies.
Moreover, the upper bound for \(T\) and the identity
\(R_{n_0}=\lambda^{n_0}\) give
\[
 T<\lambda^\alpha T_{n_0}
   =C_{\rm time}R_{n_0}^{-\alpha},
\]
which is the time restriction in Lemma~\ref{lem:barrier}.

\smallskip
\noindent\emph{Initial data and transfer of the barriers.}
Choose the signs from \eqref{eq:positive-forward-interaction}.  For
\(v\in\Gamma_+\) in the two base generations, set
\[
 \widehat\theta(0,v)=\mathbf{s}_v\varepsilon_0 w_v,
 \qquad
 \widehat\theta(0,-v)=\widehat\theta(0,v),
\]
where \(\varepsilon_0\) is chosen as in Lemma~\ref{lem:barrier}, and set all
other coefficients to zero.  These are real-valued
trigonometric-polynomial initial data.  Proposition~\ref{prop:global} gives
a global smooth solution.  The
coordinatewise convergence in its proof carries both barriers from the
finite truncations to the limit.

Since \(t_n\uparrow T_{n_0}\leq T\), every mode above generation \(n_0\)
has been activated by time \(T\).  Lemma~\ref{lem:barrier} and the choice
\(m=n_0\) in the degree estimate therefore give
\[
 -\log A_v\leq \varepsilon_{\rm rad}\abs{v}=(1-\rho)\abs{v},
 \qquad
 A_v e^{-\rho\abs{v}}\geq e^{-\abs{v}}.
\]
Consequently,
\[
 \abs{\widehat\theta(T,v)}
 \geq A_v\abs{v}^{\alpha-\gamma}e^{-\rho\abs{v}}
 \geq \abs{v}^{\alpha-\gamma}e^{-\abs{v}}.
\]

\smallskip
\noindent\emph{Upper bound for the analyticity radius.}
Choose modes \(\xi_j\) from generations \(n_j\to\infty\).  Then
\(\abs{\xi_j}\to\infty\), and
\begin{equation}\label{eq:tail}
 \abs{\widehat\theta(T,\xi_j)}
 \geq \abs{\xi_j}^{\alpha-\gamma}e^{-\abs{\xi_j}},
 \qquad
 \gamma=\frac{\log\mu}{\log\lambda}<\frac12,
\end{equation}
where \(\lambda\) is the dominant eigenvalue of \(A\), equivalently the
largest root of \(z^3-z^2-2z+1\), and \(\mu\) is the modulus of its
negative root.  At analytic weight
\(\sigma=1\), the corresponding weighted summands satisfy
\[
 e^{2\abs{\xi_j}}\abs{\widehat\theta(T,\xi_j)}^2
 \geq \abs{\xi_j}^{2(\alpha-\gamma)}.
\]
They do not tend to zero.  Hence
\(\rad(\theta(T))\leq1\).  The initial datum has finite Fourier support, so
\(\rad(\theta(0))=\infty\).

\smallskip
\noindent\emph{Positive lower bound for the analyticity radius.}
The parameter choices also ensure that
\(q_*=\overline C\varepsilon_+<1\), and hence
\(B_v=\overline C^{-1}q_*^{\ell_v}\leq1\).  The upper barrier therefore
gives
\[
 \abs{\widehat\theta(T,v)}
 \leq \abs{v}^{\alpha-\gamma}e^{-\rho\abs{v}}
 \qquad(v\in\Gamma).
\]
Each generation contains two positive modes and their negatives, and
\(\abs{v}{\approx}\lambda^n\).  Therefore, for every
\(0\leq\sigma<\rho\),
\[
 \sum_{v\in\Gamma}
 e^{2\sigma\abs{v}}\abs{\widehat\theta(T,v)}^2
 \lesssim
 \sum_{n\geq n_0-1}
 \lambda^{2(\alpha-\gamma)n}
 e^{-c(\rho-\sigma)\lambda^n}
 <\infty.
\]
Thus \(\rad(\theta(T))\geq\rho>0\), completing the proof.
\end{proof}

\section{Discussion and open problems}
\label{sec:scope}

The vector model \eqref{eq:model} combines Laplacian dissipation with
a first-order quadratic nonlinearity satisfying an exact \(L^2\)
cancellation.  Theorem~\ref{thm:main} nevertheless produces a global
smooth solution whose analyticity radius is equal to one at both
\(t=0\) and a prescribed time \(T>0\).  This conclusion concerns the
anisotropic model and its invariant cyclic Burgers family; it does
not extend directly to the genuine Navier--Stokes nonlinearity.

For projected SQG, Theorem~\ref{thm:SQG-main} constructs a symmetric
sparse set with exact additive separation and iterates the transfer in
the projected equation itself.  The solution is global and smooth, starts
from initial data with finite Fourier support, and has finite, positive
analyticity radius at the prescribed time.  The projection is tailored to
the cascade, so this
conclusion does not extend directly to the unprojected SQG equation.

Two structural questions remain.\\
{\it Structure of the Navier--Stokes interaction.}
The first direction is to identify structural properties of the
genuine Navier--Stokes nonlinearity, beyond its energy cancellation
and differential order, that enforce quantitative growth of the
analyticity radius.  Such a characterization would distinguish the
role of the full interaction geometry from what can be deduced from
the energy law alone.
\\{\it Robustness of the projected-SQG cascade.}
The second direction is to determine whether the sparse construction can
be made robust under enlargement of \(\Gamma\), or whether an analogous
coefficientwise lower bound can be obtained for less restrictive
projections.  In the present proof, exact additive separation and the
exponential suppression of feedback interactions at the input modes are
essential.  A denser set would introduce new interactions whose cumulative
size and signs are not controlled by the current barriers.

\section*{Acknowledgments}
The research of Ke Chen was supported by the Research Centre for Nonlinear Analysis at The Hong Kong Polytechnic University.
The research of Quoc-Hung Nguyen was supported by the CAS Project for
Young Scientists in Basic Research (Grant No.~YSBR-031) and by the
National Natural Science Foundation of China (Grant
Nos.~1251101538 and 12595282). Ping Zhang was partially supported by
the National Key R\&D Program of China (Grant No.~2021YFA1000800) and
the National Natural Science Foundation of China (Grant
Nos.~12421001, 12494542, and 12288201).

\section*{AI-use disclosure}

This manuscript was written by the authors.  OpenAI's ChatGPT with
GPT-5.6 was used as an assistive tool for language editing, organization,
bibliographic checks, and exploratory mathematical discussion.  The
models in Theorems~\ref{thm:main} and~\ref{thm:SQG-main} were proposed by
the authors.  During the development of the construction underlying
Theorem~\ref{thm:SQG-main}, ChatGPT also suggested preliminary ideas
through several discussions.  The authors independently assessed these
suggestions, verified all arguments, and rewrote the complete proof in
their own form.  The authors take full responsibility for the content
and conclusions of the paper.

\end{document}